\documentclass{article}
\usepackage[scale=0.7]{geometry}
\usepackage{amsmath}
\usepackage{amssymb}
\usepackage{amsfonts}
\usepackage{amsthm}
\usepackage{enumitem}
\usepackage{mathrsfs}
\usepackage{hyperref}
\hypersetup{hypertex = true, colorlinks = true, linkcolor = blue, anchorcolor = blue, citecolor =blue}
\usepackage{cleveref}
\crefname{equation}{}{}

\usepackage{graphicx}
\usepackage{float} 
\usepackage{authblk}

\newtheorem{theorem}{Theorem}
\newtheorem{lemma}[theorem]{Lemma}
\newtheorem{corollary}[theorem]{Corollary}
\newtheorem{proposition}[theorem]{Proposition}
\newtheorem{definition}[theorem]{Definition}
\newtheorem{assumption}[theorem]{Assumption}
\newtheorem{remark}{Remark}

\newcommand{\E}{\mathbb{E}}

\newcommand{\x}{\boldsymbol{x}}
\newcommand{\z}{\boldsymbol{z}}
\newcommand{\y}{\boldsymbol{y}}
\newcommand{\R}{\mathbb{R}}
\newcommand{\bm}{\boldsymbol}
\newcommand{\abs}[1]{\left| #1 \right|}
\title{Randomized Quasi-Monte Carlo and Importance Sampling for Super-Fast Growing Functions with Applications to Finance}  
\author[1]{Jianlong Chen}
\author[2]{Yu Xu \thanks{Corresponding author: yu-xu22@mails.tsinghua.edu.cn}}
\author[3]{Jiarui Du}
\author[4]{Xiaoqun Wang}

\affil[1,2,4]{Department of Mathematical Sciences, Tsinghua University, Beijing, China}
\affil[3]{School of Mathematics, South China University of Technology, Guangzhou, China}

\date{\today}

\begin{document}

\maketitle

\begin{abstract}
Many problems can be formulated as high-dimensional integrals of discontinuous functions that exhibit significant boundary growth, challenging the error analysis and applications of randomized quasi-Monte Carlo (RQMC) methods. This paper studies RQMC methods for super-fast growing functions satisfying generalized exponential growth conditions, with a special focus on financial derivative pricing.
The main contribution of this paper is threefold. First, by combining RQMC with importance sampling (IS), we derive a new error bound for a class of integrands, whose values and derivatives are bounded by the critical growth function $e^{A|\boldsymbol{x}|^2}$ with $A = 1/2$. This result extends the existing results in the literature, which are limited to the case $A < 1/2$. We demonstrate that by imposing a light-tailed condition on the proposal distribution of IS, RQMC can achieve an error rate of  $O(n^{-1  + \epsilon})$ with a sample size n and an arbitrarily small $\epsilon >0$. 
Second, we verify that the Gaussian proposals used in Optimal Drift Importance Sampling (ODIS) satisfy the required light-tailed condition, providing a rigorous theoretical guarantees for RQMC-ODIS in critical growth scenarios.
Third, for discontinuous integrands from finance, we prove that the integrands after preintegration satisfy the exponential growth condition. This ensures that the preintegrated functions can be seamlessly incorporated into our RQMC-IS framework.
Numerical experiments on financial derivative pricing validate our theory, showing that the RQMC-IS with preintegration is effective in handling problems with discontinuous payoffs, successfully achieving the expected convergence rates.

Keywords:Importance sampling, Quasi-Monte Carlo, Growth condition, Preintegration
\end{abstract}


% \begin{highlights}
% \item 
% \item 
% \item 
% \end{highlights}

% Keywords

\maketitle

\section{Introduction}\label{sec:introduction}

Computational finance often faces the challenge of high-dimensional numerical integration, especially in areas such as derivative valuation and risk measurement \cite{Glasserman2004}. A common problem involves estimating the expectations of the form $\mathbb{E}[f(\boldsymbol{X})]$, where $\boldsymbol{X} = (X_1, \ldots, X_d)$ is a $d$-dimensional standard normal random vector, and $f: \mathbb{R}^d \to \mathbb{R}$. In many practical scenarios, $f$ exhibit discontinuities or is defined on unbounded domains. Monte Carlo (MC) method offers a flexible approach for such high-dimensional problems, but it suffers from a relatively slow theoretical convergence rate of $O(n^{-1/2})$ with a sample size $n$, which limit its efficiency when high precision is required.

Quasi-Monte Carlo (QMC) \cite{Caflisch1998, Caflisch1997, dick2010, niederreiter1992, owenqmc}, utilizing low-discrepancy point sets \(\{\boldsymbol{y}_j\}_{j=1}^n\) in \([0,1)^d\), provides a promising alternative. These point sets are designed to fill the unit cube more uniformly than random points, leading to superior empirical performance and a theoretically faster convergence rate, potentially approaching \(O(n^{-1+\epsilon})\) with a sample size n and an arbitrarily small $\epsilon >0$ for sufficiently smooth integrands \cite{owen2006}. Randomized quasi-Monte Carlo (RQMC) methods combine the deterministic sampling of QMC with randomization, preserving the high accuracy of QMC with the error estimation advantages of MC. Common randomization methods include random shift \cite{random_shift} and scramble \cite{owen1995scramble}. By suitable transformation, 
the expectation $\mathbb{E}[f(\bm X)]$ can be changed to the integration over the unit cube 
\begin{equation}
\mathbb{E}[f(\bm X)] = \mathbb{E}[h(\boldsymbol{Y})] = \int_{[0,1)^d} h(\boldsymbol{y})  \mathrm{d}\boldsymbol{y}, \label{term:QMCintegral}
\end{equation}
 and the QMC (RQMC) estimator is \(\widehat{I}_n(h) = \frac{1}{n} \sum_{j=1}^n h(\boldsymbol{y}_j).\)

However, the effectiveness of QMC (RQMC) depends on the properties of the integrand, such as the smoothness and boundedness. Thus two major challenges arise: (1) the function \(f\) arising in finance often contains discontinuities, which violate the smoothness conditions needed for QMC (RQMC); (2) even when \(f\) is smooth, the integrand \(h\) in \cref{term:QMCintegral}  usually has the form \(f\circ\Phi^{-1}\), where \(\Phi\) denotes the cumulative distribution function of the standard normal distribution \(N(0,1)\). The transformation \(\Phi^{-1}\) can introduce singularities or rapid growth near the boundary of the unit cube, while functions like \(e^{\frac{1}{2}|\boldsymbol{x}|^2}\)($|\x|$ denotes the Euclidean norm of $\x$) become unbounded, complicating theoretical analysis.

To overcome the difficulty of discontinuities in the integrands, researchers have developed preprocessing techniques such as the \textit{preintegration}  \cite{preintegration_kuo2013, preintegration_kuo2017, preintegration_kuo2010, preintegration_kuo2018}, which is also called conditional Monte Carlo method \cite{CMC_simulation}. This method involves integrating out one variable (say \(x_j\)) conditioned on the others (denoted by \(\boldsymbol{x}_{-j}\)), defining a smoothed conditional expectation \(P_j f(\boldsymbol{x}_{-j}) = \mathbb{E}[f(\boldsymbol{x}) | \boldsymbol{x}_{-j}]\). For function of the form \( f(\x)=g(\x )\mathbb{I} \{ \phi (\x )\geq 0\}\)( $\mathbb{I}(\cdot)$ denotes the indicator function), under suitable monotonicity and smoothness assumptions on the discontinuity boundary \(\phi(\boldsymbol{x}) = 0\) and on the continuous factor \(g(\boldsymbol{x})\), He \cite{he2019} demonstrated that \(P_j f\) can inherit significant smoothness from \(g\) and \(\phi\).

Importance Sampling (IS) \cite{dick2019, glasserman1999, kuk1999, owen2000, zcj2021} is a variance reduction method in MC. IS introduces a proposal density \(q(\boldsymbol{x})\) to rewrite the expectation
\begin{equation}
\mathbb{E}[f(\boldsymbol{X})] = \int_{\mathbb{R}^d} f(\x)\varphi(\x) d\x = \int_{\mathbb{R}^d} \frac{f(\x)\varphi(\x)}{q(\x)} q(\x) d\x := \mathbb{E}_q\left[f_{IS}(\boldsymbol{X})\right] , \notag
\end{equation}
where \( f_{IS}(\boldsymbol{x}) = \frac{\varphi(\boldsymbol{x})f(\boldsymbol{x})}{q(\boldsymbol{x})}\), and \(\varphi\) is the standard normal density. Then QMC (RQMC)-IS estimator is  \(\widehat{I}_n^{IS}(f) = \frac{1}{n} \sum_{j=1}^n f_{IS} \circ G^{-1}(\boldsymbol{y}_j)\), where \(G^{-1}\) is the inverse of the cumulative distribution function (C.D.F.) of \(q\). There is a lot of theoretical analysis on MC-IS, but few convergence results on QMC-IS. Recently, He et al. \cite{zhengzhan} and Ouyang et al. \cite{Ouyang2024} analyzed the convergence results of RQMC-IS. This paper builds on these existing results to perform a deeper analysis and get more general results.

A popular choice for the proposal in derivative pricing is the Optimal Drift Importance Sampling (ODIS), which selects the proposal \(q(\boldsymbol{x}) = \varphi(\boldsymbol{x} - \boldsymbol{\mu}^*)\) from the family of shifted normal densities \(N(\boldsymbol{\mu}, \boldsymbol{I}_d)\) where \(\boldsymbol{\mu}^* = \arg \max \varphi(\boldsymbol{z})f(\boldsymbol{z})\). Glassman \cite{Glasserman2004} studied the convergence of MC-ODIS and its applications in finance, confirming the effectiveness of ODIS in option pricing, risk management, and other tasks under various financial models. He et al. \cite{zhengzhan} obtained the convergence order of RQMC-ODIS under the so-called boundary growth condition proposed by Owen \cite{owen2006}.

However, the theory of RQMC-IS still has many limitations and challenges. Recent work by Ouyang et al. \cite{Ouyang2024} established the convergence rates for RQMC-IS under the condition that \(f_{IS} = f\varphi/q\) belongs to a function class \(\mathcal{S}(\mathbb{R}^d)\) satisfying a growth condition 
\begin{equation*}
    \sup \limits_{\boldsymbol{u}\subseteq\{1,\cdots,d\}} |\partial^{\boldsymbol{u}} f_{IS}(\boldsymbol{x})| \leq C e^{A|\boldsymbol{x}|^k}
\end{equation*}
with \(k=2\) and \(A < 1/2\). 
Critically, as highlighted in Section \ref{sec:ImportanceSampling} of this paper, an example $f_{IS} = 1/\varphi \propto e^{\frac{1}{2}|\boldsymbol{x}|^2}$ arising in ODIS for constant $f$ violates this key condition, since the derivatives of $f_{IS}$ grow too rapidly to be bounded by $e^{A|\boldsymbol{x}|^2}$ with $A<1/2$ (cannot even bounded by $e^{\frac{1}{2}|\boldsymbol{x}|^2}$).
This paper bridges this theoretical gap. We first extend the convergence theory for RQMC-IS to a generalized growth condition of the form \(Ce^{A|\boldsymbol{x}|^2 + B|\boldsymbol{x}|}\), accommodating the critical \(A=1/2\) case. Specifically, we show that ODIS satisfies the conditions of our theorem. Building on this, we present and prove a theorem demonstrating that functions after preintegration also satisfy our growth condition. We thus establish that our results are applicable to problems in finance.

More specifically, our first key contribution is a generalized convergence theory for RQMC-IS. We define a broader growth condition of the form \(Ce^{A|\boldsymbol{x}|^2 + B|\boldsymbol{x}|}\) for \(f_{IS}\), which enables us to prove the convergence rates for the RQMC-IS estimator \(\widehat{I}_n^{IS}(f)\). For \(A < 1/2\), we recover the near \(O(n^{-1})\) rate under a mild condition (see \ref{condition:a} in Section \ref{sec:MainResults}) on the proposal \(q\). Crucially, for the critical case \(A = 1/2\), we prove a convergence rate under a light-tailed condition (see \ref{condition:b} in Section \ref{sec:MainResults}) on \(q\), circumventing the \(A < 1/2\) barrier of previous work. 
Second, we demonstrate the applicability of our theory to ODIS. We formally show that the Gaussian proposals \(q(\boldsymbol{x}) = \varphi(\boldsymbol{x} - \boldsymbol{\mu})\) used in ODIS satisfy the required light-tailed condition (see Theorem  \ref{prop:ODIS}).
Finally, we prove that functions after preintegration satisfy our growth condition. We analyze the preintegration operator \(P_j\) applied to discontinuous functions of the form \(f(\boldsymbol{x}) = g(\boldsymbol{x})\mathbb{I}_{\{\phi(\boldsymbol{x}) \geq 0\}}\). Building on He's results \cite{he2019}, we show that if the original components \(g\) and \(\phi\) satisfy exponential growth conditions (see Assumption \ref{assum:grc}), the smoothed function \(P_j f\) also satisfies our generalized exponential growth condition (see Theorem \ref{thm:pvf}). This enables the application of our RQMC-IS convergence theory to the preintegrated functions, providing a complete theoretical framework for handling discontinuous payoffs using a combination of preintegration, IS (like ODIS), and RQMC.

The remainder of this paper is organized as follows. Section \ref{sec:ImportanceSampling} reviews importance sampling and the limitations of the existing RQMC convergence theory, particularly with regard to ODIS. Section \ref{sec:MainResults} presents our main convergence results for RQMC-IS under generalized growth conditions and demonstrates the applicability of these results to ODIS. Section \ref{sec:Preintegration} details the preintegration technique and proves that it preserves exponential growth rates under suitable assumptions. Section \ref{sec:NumericalExperiments} applies our theory to option pricing problems and presents numerical results. Section \ref{sec:Conclusion} concludes and discusses future research directions. Proofs of supporting lemmas are provided in Appendix \ref{sec:Appendix}.

\section{Preliminary and Previous Work}\label{sec:ImportanceSampling}
\subsection{Importance Sampling}
Many problems in financial engineering can be reduced to calculating the expectation \(\E[f(\boldsymbol{X})]\), where $\boldsymbol{X}$ is a d-dimensional standard normal random vector. The problems can be transformed into an integral over the unit cube, i.e.,
\begin{equation*}
\mathbb{E}[f(\boldsymbol{X})] = \int_{[0,1)^d} f \circ \Phi^{-1}(\boldsymbol{y}) \mathrm{d}\boldsymbol{y},
\end{equation*}
where \(\Phi^{-1}(\cdot)\) is the inverse of the C.D.F. of \(N(0,1)\), acting on \(\boldsymbol{y}\) component-wisely. To estimate the expectation, RQMC quadrature rule of the following form can be used
\begin{equation}\label{term:StandardEstimator}
\widehat{I}_n(f) = \frac{1}{n} \sum_{j=1}^{n} f \circ \Phi^{-1}(\boldsymbol{y}_j),
\end{equation}
where $\{ \boldsymbol{y}_j \}$ is a randomized low discrepancy point set in $[0,1)^d$. 

Importance sampling is a variance reduction method in MC methods by choosing a suitable importance density. Formally, let $\varphi$ be the density of the d-dimensional standard normal distribution and $q$ be another density. Then we can rewrite \(\E[f(\boldsymbol{X})]\) as
\begin{equation}\label{term:IntegralOfIS}
\mathbb{E}[f(\bm X)] = \int_{\mathbb{R}^d} f(\boldsymbol{x})\varphi(\boldsymbol{x}) \mathrm{d}\boldsymbol{x} = \int_{\mathbb{R}^d} \frac{\varphi(\boldsymbol{x}) f(\boldsymbol{x})}{q(\boldsymbol{x})} q(\boldsymbol{x}) \mathrm{d}\boldsymbol{x} := \E_q[f_{IS}(\boldsymbol{X})], 
\end{equation}
where \(f_{IS} := \frac{\varphi f}{q}\) and \( \boldsymbol{X} \) follows the distribution corresponding to density $q$. Then $G^{-1}(\boldsymbol{U})$ follows the distribution corresponding to density $q$ when $\boldsymbol{U}$ follows the uniform distribution on $[0,1)^d$, where \(G^{-1}\) is the inverse of the C.D.F. of \(q\). Thus the importance sampling estimator is 
\begin{equation}\label{eq:IS_estimator}
\widehat{I}^{IS}_n(f) = \frac{1}{n} \sum_{j=1}^n f_{IS} \circ G^{-1}(\boldsymbol{y}_j),
\end{equation}
where $\{ \boldsymbol{y}_j \}$ is a randomized low discrepancy point set in $[0,1)^d$. 

 The selection of the proposal is crucial. Note that the original density is the standard normal distribution \( N(0, I_d) \). A popular method called Optimal Drift Importance Sampling (ODIS) chooses the proposal from the normal distribution family \( N(\mu, I_d) \). Formally, denote the support set of $f$ as $D$, then the density of the ODIS-proposal is given by \cite{glasserman1999} 
\begin{equation}
q(\boldsymbol{z}) = (2\pi)^{-d/2} \exp \left[ -\frac{1}{2} (\boldsymbol{z} - \boldsymbol{\mu}^\star)^T (\boldsymbol{z} - \boldsymbol{\mu}^\star) \right], \label{ODIS_proposal}
\end{equation}
where
\begin{equation*}
\boldsymbol{\mu}^\star = \arg \max_{\boldsymbol{z} \in D} \varphi(\boldsymbol{z})f(\boldsymbol{z}).\notag
\end{equation*}
The ODIS is widely used in practice, both in MC and QMC setting. However , the theoretical foundation of QMC-based ODIS is not yet fully established. He et al. \cite{zhengzhan} proved the convergence rate of QMC-ODIS under Owen's boundary growth condition \cite{owen2006} when using scrambled nets. We will study the convergence of RQMC-ODIS under a more general framework.
\subsection{Convergence rate of importance sampling for RQMC methods}
Below we review several related results on QMC-IS.
For convenience, we introduce some notations. Denote $1{:}d=\{1,2,\dots,d\}$ and
$\partial_j\phi:=\partial \phi/\partial x_j$. For $\boldsymbol{u}\subseteq 1{:}d$, $\partial^{\boldsymbol{u}} \phi$ denotes the derivative taken with respect to each $x_j$ once  for all $j\in \boldsymbol{u}$.  For any multi-index $\bm \alpha=(\alpha_1,\dots,\alpha_d)$ whose components are nonnegative integers, denote
\begin{equation*}
(\partial^{\bm \alpha} \phi)(\bm x):=\frac{\partial^{\abs{\bm \alpha}}\phi}{\partial x_1^{\alpha_1}\dots x_d^{\alpha_d}}(\bm x),
\end{equation*}
where $\abs{\bm\alpha}=\sum_{i=1}^{d}\alpha_i$. If $\alpha_i=1$ for all $i\in \boldsymbol{u}$ and $\alpha_i=0$ otherwise, then $\partial^{\bm \alpha} \phi=\partial^{\boldsymbol{u}} \phi$. 

\begin{definition}
A function $f(\boldsymbol{x})$ defined over $\mathbb{R}^d$ is called a smooth function if for any $\boldsymbol{u} \subseteq 1{:}d$, $\partial^{\boldsymbol{u}} f$ is continuous. Let $\mathcal{S}(\mathbb{R}^d)$ be the class of such smooth functions.   
\end{definition}
Recently, Ouyang et al. \cite{Ouyang2024} defined the growth conditions for the functions in $\mathcal{S}(\mathbb{R}^d)$
and established the following theorem by a projection method for this function class.
\begin{proposition}\label{prop:ProjectionForStandardRQMC}
Let \(\{\boldsymbol{y}_1, \ldots, \boldsymbol{y}_n\}\) be a RQMC point set used in the estimator \(\widehat{I}_n(f)\) given by \cref{term:StandardEstimator} such that each \(\boldsymbol{y}_j \sim U[0,1]^d\) and
\[ D_n^*(\{\boldsymbol{y}_1, \ldots, \boldsymbol{y}_n\}) \leq C \frac{(\log n)^{d-1}}{n} \quad \text{a.s.,} \]
where \(C\) is a constant independent of \(n\). For any function $f \in \mathcal{S}(\mathbb{R}^d)$, if it satisfies that
\begin{equation}
\sup \limits_{\boldsymbol{u} \subseteq 1:d} |\partial^{\boldsymbol{u}} f(\boldsymbol{x})| \leq C e^{A|\boldsymbol{x}|^k} , \forall \boldsymbol{x} \in \mathbb{R}^d \label{term:GrowthCondition}
\end{equation}
 with \(A>0, C>0\) and \(0 < k < 2 \), then we have
\begin{equation}
\mathbb{E} \left[ \left( \widehat{I}_n(f) - \mathbb{E}[f(\boldsymbol{X})] \right)^2 \right] = O \left( n^{-2} (\log n)^{3d-2} \exp \left\{ 2A (8d \log n)^{\frac{k}{2}} \right\} \right), \notag
\end{equation}
where the constant in \(O\) depends on $A,C,k$ and $d$.
\end{proposition}

Moreover, similar result can be established for QMC-IS. Suppose that $q$ in \cref{term:IntegralOfIS} has the form
\begin{equation*}
q(\boldsymbol{x}) = \prod_{j=1}^d q_j(x_j) 
\end{equation*}
and let $G_j$ be the C.D.F. of $q_j$. Then $G^{-1}(\boldsymbol{y})$ can be written as $(G_1^{-1}(y_1),\cdots,G_d^{-1}(y_d))$. We have the following theorem for RQMC-IS \cite{Ouyang2024}.

\begin{proposition}\label{prop:ProjectionForIS_A<1/2}
Assume $A < 1/2$. If the RQMC point set $\{ \boldsymbol{y}_1, \ldots, \boldsymbol{y}_n \}$ satisfies the conditions in Proposition \ref{prop:ProjectionForStandardRQMC}. For any function \(f\) in \(\mathcal{S}(\mathbb{R}^d)\), if the proposal $q$ satisfies that $\mathbb{E}_q \left[ |\boldsymbol{X}|^2 \right] < \infty, f/q \in \mathcal{S}(\mathbb{R}^d)$ and
\begin{equation}
\sup \limits_{\boldsymbol{u} \subseteq 1:d} |\partial^{\boldsymbol{u}} (\frac{f}{q})(\boldsymbol{x})| \leq C e^{A|\boldsymbol{x}|^2} , \forall \boldsymbol{x} \in \mathbb{R}^d, \label{term:GrowthCondition_IS}
\end{equation}
then for the RQMC-IS estimator $\widehat{I}^{IS}_n(f)$ given in \cref{eq:IS_estimator} we have
\begin{equation*}
\mathbb{E} \left[ \left( \widehat{I}_n^{IS}(f) - \mathbb{E} \left[ f(\boldsymbol{X}) \right] \right)^2 \right] = O \left( n^{-2} (\log n)^{3d-2} \right),
\end{equation*}
where the constant in \(O\) depends on $A,C,k,d$ and $\mathbb{E}_q \left[ |\boldsymbol{X}|^2 \right]$.
\end{proposition}
A fundamental limitation of Proposition \ref{prop:ProjectionForIS_A<1/2} is that it requires $A< \frac{1}{2}$ in the condition \cref{term:GrowthCondition_IS}.
Due to this limitation, Proposition \ref{prop:ProjectionForIS_A<1/2} is not applicable to the ODIS. For instance, when applying ODIS to the constant function \( f \equiv 1 \), the optimal proposal distribution given in \cref{ODIS_proposal} reduces to the original standard normal distribution \( q(\boldsymbol{x}) = \varphi(\boldsymbol{x}) \). In this case, the function $f/g$ in the growth condition \cref{term:GrowthCondition_IS} is
\begin{equation}\label{term:ExampleForM}
\frac{f(\boldsymbol{x})}{q(\boldsymbol{x})} = \frac{1}{\varphi(\boldsymbol{x})} \propto e^{\frac{1}{2}|\boldsymbol{x}|^2},
\end{equation}
which exhibits an exponential growth rate with \( A= 1/2 \) . This violates the condition \( A < 1/2 \) required by Proposition \ref{prop:ProjectionForIS_A<1/2}.

Furthermore, even the growth condition \cref{term:GrowthCondition_IS} with \( A = 1/2 \)  would be insufficient. The partial derivative of \( e^{\frac{1}{2}|\boldsymbol{x}|^2} \) contains a term \( x_j e^{\frac{1}{2}|\boldsymbol{x}|^2} \) for some $j \in 1{:}d$, which grows faster than \( e^{\frac{1}{2}|\boldsymbol{x}|^2} \) and cannot be satisfied the condition in \cref{term:GrowthCondition_IS} with $A = 1/2$. This highlights a fundamental limitation of the existing theoretical framework in handling ODIS.

\section{RQMC-IS Convergence under a Generalized Growth Condition}\label{sec:MainResults}

To address the theoretical limitations discussed above, we now develop a new convergence theory for RQMC-IS. Our framework is built upon a more general growth condition that accommodates the critical case where the integrand and its derivatives exhibit a growth proportional to \(e^{\frac{1}{2}|\boldsymbol{x}|^2}\).

We consider the case where \( k = 2 \) and define a growth condition on the function \(f/q\). We assume that there exist constants \(A,B,C>0\) such that for all \( \boldsymbol{x} \in \mathbb{R}^d \),
\begin{equation}
\sup_{\boldsymbol{u} \subseteq 1:d} |\partial^{\boldsymbol{u}} (f(\boldsymbol{x})/q(\boldsymbol{x}))| \leq  Ce^{A|\boldsymbol{x}|^2+B|\boldsymbol{x}|}. \label{term:GrowthCondition_Extended}
\end{equation}

To present our results, we introduce two possible conditions on the proposal density \(q\) as follows.
\begin{enumerate}[label = \textbf{Condition}~\Roman*, leftmargin=*] 
  \item 
  The proposal \(q\) has a finite second moment, i.e., \(\E_q[|\boldsymbol{X}|^2] < \infty \). \label{condition:a}
   \item \textbf{(Light-Tailed Condition)} 
   The proposal \(q\) has tails that decay sufficiently fast. Specifically, for the case \(A=1/2\) and $B$ in \cref{term:GrowthCondition_Extended}, there exist constants $\alpha>0, \beta>0, \gamma>1$ and $C_0>0$, such that for any $R>1$,
    \[ \mathbb{E}_q \left[|\boldsymbol{X}|^4 e^{2B|\boldsymbol{X}|} \mathbb{I}_{\{|\boldsymbol{X}| \geq R - 1\}} \right] \leq C_0 R^{\alpha} e^{-\beta R^\gamma}. \] \label{condition:b}
\end{enumerate}
\begin{remark}
    \ref{condition:a} is trivial and easily satisfied. When the growth rate of $f/g$ and its derivatives do not grow so fast, it only requires that $q$ satisfies the existence of the second moment. However, when $f/g$ and its derivatives grow too fast and approach the critical state, we need \ref{condition:b}, i.e., the light-tailed condition of the density, to be satisfied in order to obtain our convergence conclusion. This light-tailed condition can be regarded as a generalization of the normal density, and we will also prove that the ODIS proposal \cref{ODIS_proposal} satisfies this condition in Theorem \ref{prop:ODIS}.

\end{remark}
Now we present our main convergence results for the RQMC-IS estimator.

\begin{theorem}\label{thm:MainTheorem}
Assume that \(f/q \in \mathcal{S}(\mathbb{R}^d) \) satisfies the growth condition \cref{term:GrowthCondition_Extended} and that the proposal $q$ satisfies one of the above conditions (\ref{condition:a} or \ref{condition:b}), then the following convergence results for the RQMC-IS estimator $\widehat{I}^{IS}_n(f)$ hold.
\begin{enumerate}[label=(\alph*)]
    \item When $A < \frac{1}{2}$, if $q$ satisfies the \ref{condition:a}, we have
    \begin{equation*}
    \mathbb{E} \left[ ( \widehat{I}_n^{IS}(f) - \mathbb{E} \left[ f(\boldsymbol{X}) \right] )^2 \right] = O \left( n^{-2} r_1(n) \right),
    \end{equation*}
    where $r_1(n) = (\log n)^{3d-2}$.
    \item When $A = \frac{1}{2}$, if $q$ satisfies the \ref{condition:b}, we have 
    \begin{equation*}
    \mathbb{E} \left[ ( \widehat{I}_n^{IS}(f) - \mathbb{E} \left[ f(\boldsymbol{X}) \right] )^2 \right] = O \left( n^{-2} r_2(n) \right),
    \end{equation*}
    where $r_2(n) = (\log n)^{4d/\gamma + 2d - 2}e^{\sqrt{d}B(2/\beta)^{1/\gamma} (\log n)^{1/\gamma}}$.
\end{enumerate}
\end{theorem}

Before the proof of Theorem \ref{thm:MainTheorem}, we introduce two key lemmas. Their proofs are deferred to the Appendix \ref{sec:Appendix} to maintain the flow of the main argument.

\begin{lemma}\label{lemma:V_HK}
Assume that \( f/q \in \mathcal{S}(\mathbb{R}^d)\) and that there exists \(B,C>0\) such that
\begin{equation}
\sup_{\boldsymbol{u} \subseteq 1:d} |\partial^{\boldsymbol{u}} \frac{f}{q}(\x)| \leq  Ce^{\frac{1}{2}|\x|^2+B|\boldsymbol{x}|} \text{ for all } \x \in \mathbb{R}^d. \notag
\end{equation}
Define a projection operator $\hat{P_R}: \mathbb{R} \to \mathbb{R}$ by 
\begin{equation*}\label{eq:smoothprojection}
        \hat{P_R}(x) = \left\{\begin{array}{ll}
            -R+\frac{1}{2}, \ &\ x\in [-\infty,-R]\\
            \frac{1}{2}x^2 + Rx +\frac{(R-1)^2}{2},
            \ &\ x\in(-R,-R+1)\\
            x,\ &\ x\in [-R+1,R-1]\\
            -\frac{1}{2}x^2 + Rx - \frac{(R-1)^2}{2},\ &\ x\in (R-1,R)\\
            R-\frac{1}{2},\ &\ x\in [R,\infty]
        \end{array}\right. 
    \end{equation*}
and define $P_R: \mathbb{R}^d \to \mathbb{R}^d$ as $P_R(\x) = (\hat{P_R}(x_1),\dots,\hat{P_R}(x_d))$.
Then for any \(R>1\), the Hardy-Krause variation \cite{niederreiter1992} of \( f_{IS} \circ P_R \circ G^{-1} \) is bounded by
\[
V_{HK}(f_{IS} \circ P_R \circ G^{-1}) \leq B (2 \sqrt{d} + 3)^d R^{2d} e^{\sqrt{d} B R},
\]
\end{lemma}

\begin{lemma}\label{lemma:ProjectionError}
If \( f/q \in \mathcal{S}(\mathbb{R}^d)\) and there exist \(B,C>0\) such that
\begin{equation}
\sup_{\boldsymbol{u} \subseteq 1:d} |\partial^{\boldsymbol{u}} \frac{f}{q}(\x)| \leq  Ce^{\frac{1}{2}|\x|^2+B|\boldsymbol{x}|} \text{ for all } \x \in \mathbb{R}^d \notag
\end{equation} 
and \( q \) satisfies the \ref{condition:b}, then for any \( R>1 \), we have
\begin{equation}
\mathbb{E}_q \left| f_{IS}(\boldsymbol{X}) - f_{IS} \circ P_R(\boldsymbol{X}) \right|^2 \leq C_1 R^{\alpha} e^{-\beta R^\gamma}, \notag
\end{equation}
where the constant \(C_1\) is independent of \(R\) and the operator $P_R$ is defined in Lemma \ref{lemma:V_HK}.
\end{lemma}

With the help of Lemma \ref{lemma:V_HK} and Lemma \ref{lemma:ProjectionError}, we may prove the theorem.
\begin{proof}[Proof of Theorem \ref{thm:MainTheorem}]
(a) For the case $A < 1/2$, we just need to choose an \(\epsilon > 0 \) small enough such that \(A+\epsilon < \frac{1}{2}\).  Then there exists a constant $C'$ for which
$
Ce^{A|\boldsymbol{x}|^2+B|\boldsymbol{x}|} \leq C'e^{(A+\epsilon)|\x|^2}.
$
This satisfies the condition of Proposition \ref{prop:ProjectionForIS_A<1/2} with $k=2$. Under the \ref{condition:a} ($\E_q[|\boldsymbol{X}|^2] < \infty$), the claimed error rate follows directly by Proposition \ref{prop:ProjectionForIS_A<1/2}.

(b) For the case $A = 1/2$, we have
\begin{align*}
\mathbb{E}\left[\left|\hat{I}_n^{IS}(f) - \mathbb{E}[f(\boldsymbol{X})]\right|^2\right] 
&\leq 3\left\{\mathbb{E}\left[\left|\hat{I}_n^{IS}(f) - \hat{I}_n^{IS, R}(f)\right|^2\right] + \mathbb{E}\left[\left|\hat{I}_n^{IS, R}(f) - \mathbb{E}_q\left[f_{IS}\circ P_R(\boldsymbol{X})\right]\right|^2\right] \right. \\
&\quad \left. + \left|\mathbb{E}_q\left[f_{IS}\circ P_R(\boldsymbol{X})\right] - \mathbb{E}_q\left[f_{IS}(\boldsymbol{X})\right]\right|^2\right\},
\end{align*}
where 
\begin{equation}
    \widehat I_n^{IS,R}(f) := \frac{1}{n} \sum_{j = 1}^n f_{IS}\circ P_R \circ G^{-1} (\boldsymbol{y}_j)\notag .
\end{equation}.

By the Cauchy inequality
\begin{align*}
\mathbb{E}\left[\left|\hat{I}_n^{IS}(f) - \hat{I}_n^{IS, R}(f)\right|^2\right]
& = \mathbb{E}_q\left[\left|\frac{1}{n}\sum_{j=1}^n\left[f\left(\boldsymbol{X}_j\right) - f \circ P_R\left(\boldsymbol{X}_j\right)\right]\right|^2\right] \\
& \leq \mathbb{E}_q\left[\left|f_{IS}(\boldsymbol{X}) - f \circ P_R(\boldsymbol{X})\right|^2\right].
\end{align*}

Moreover, we have 
\begin{align*}
\left|\mathbb{E}_q\left[f_{IS}\circ P_R(\boldsymbol{X})\right] - \mathbb{E}_q\left[f_{IS}(\boldsymbol{X})\right]\right|^2 \leq \mathbb{E}_q\left[\left|f_{IS}(\boldsymbol{X}) - f\circ P_R(\boldsymbol{X})\right|^2\right].
\end{align*}

By the Koksma-Hlawka inequality
\begin{align*}
\mathbb{E}\left[\left|\hat{I}_n^{IS, R}(f) - \mathbb{E}_q\left[f_{IS}\circ P_R(\boldsymbol{X})\right]\right|^2\right] \leq \left(V_{HK}\left(f_{IS}\circ P_R\circ G^{-1}\right) \cdot C\frac{(\log n)^{d-1}}{n}\right)^2.
\end{align*}

Thus we conclude that
\begin{align*}
& \mathbb{E}\left[\left|\hat{I}_n^{IS}(f) - \mathbb{E}[f(\boldsymbol{X})]\right|^2\right] \\
\leq & 6\mathbb{E}_q\left[\left|f_{IS}(\boldsymbol{X}) - f\circ P_R(\boldsymbol{X})\right|^2\right] + 3 C^2\left(V_{HK}\left(f_{IS}\circ P_R\circ G^{-1}\right) \cdot \frac{(\log n)^{d-1}}{n}\right)^2 \\
=: & T_1 + T_2.
\end{align*}

Using Lemma \ref{lemma:V_HK} and Lemma \ref{lemma:ProjectionError}, we can bound these two terms. For the first term $T_1$, we have
$
T_1 \le C_1 R^{\alpha} e^{-\beta R^\gamma}
$
by Lemma \ref{lemma:ProjectionError}. For the second term $T_2$, we obtain that
$
T_2 \le C_2 \frac{(\log n)^{2d-2}}{n^2} R^{4d} e^{2\sqrt{d} B R}
$
for some constant $C_2$ by Lemma \ref{lemma:V_HK}. Thus we have
\[
\mathbb{E}\left[\left|\hat{I}_n^{IS}(f) - \hat{I}_n^{IS, R}(f)\right|^2\right] \le C_1 R^{\alpha} e^{-\beta R^\gamma} + C_2 \frac{(\log n)^{2d-2}}{n^2} R^{4d} e^{2\sqrt{d} B R}.
\]
We set 
\[
R = \left(\frac{2 \log n}{\beta}\right)^{1/\gamma}.
\]
With this choice of $R$, the first term $T_1$ can be bounded as
\[
T_1 \le C_1 \left(\frac{2 \log n}{\beta}\right)^{\alpha/\gamma} e^{-2\log n} = C_1 \left(\frac{2}{\beta}\right)^{\alpha/\gamma} \frac{(\log n)^{\alpha/\gamma}}{n^2}.
\]
The second term $T_2$ can be bounded as
\begin{align*}
T_2 &\le C_2 \frac{(\log n)^{2d-2}}{n^2} \left(\frac{2 \log n}{\beta}\right)^{4d/\gamma} \exp\left(2\sqrt{d} B \left(\frac{2 \log n}{\beta}\right)^{1/\gamma}\right) \\
&= \frac{C_2(2/\beta)^{4d/\gamma}}{n^2} (\log n)^{2d-2+4d/\gamma} \exp\left(2\sqrt{d}B(2/\beta)^{1/\gamma} (\log n)^{1/\gamma}\right).
\end{align*}
Since $\gamma > 1$, the exponential factor $\exp(C_4 (\log n)^{1/\gamma})$ grows slower than any polynomial in $n$, but faster than any polylogarithmic term. Therefore, the second term $T_2$ dominates the error. The overall error is determined by the asymptotic behavior of $T_2$, which gives the rate as stated in the theorem.
\end{proof}

Having established our main convergence results, we now demonstrate their applicability to the ODIS method introduced in the previous section.

\begin{theorem}\label{prop:ODIS}
If \(q\) has the form
\begin{equation}
q(\boldsymbol{z}) = (2\pi)^{-d/2} \exp \left\{ -\frac{1}{2} (\boldsymbol{z} - \boldsymbol{\mu})^T (\boldsymbol{z} - \boldsymbol{\mu}) \right\}, \notag
\end{equation}
where  \(\bm \mu \in \mathbb{R}^d\), then \(q\) satisfies the \ref{condition:b}. Furthermore, if \(\frac{f}{q} \in \mathcal{S}(\mathbb{R}^d) \) and satisfies the growth condition \cref{term:GrowthCondition_Extended}, then for the RQMC-IS estimator  we have 
    \begin{equation}\label{eq:ODIS_error}
    \mathbb{E} \left[ ( \widehat{I}_n^{IS}(f) - \mathbb{E} \left[ f(\boldsymbol{X}) \right] )^2 \right] = O \left( n^{-2+\epsilon} \right).
    \end{equation}
\end{theorem}
\begin{proof}
For any $B>0$, consider the integral
\[
\int_{\mathbb{R}^d \setminus H} |\boldsymbol{x}|^4 \exp\{2B|\boldsymbol{x}|\} q(\boldsymbol{x}) \mathrm{d}\boldsymbol{x}.
\]
where  $H = B(0, R-1) := \left\{ \x : |\x | < R-1 \right\}$.

For $R > 4(|\mu| + 2B) + 1$, we have
\[
|\boldsymbol{x}| \geq R - 1 \implies |\boldsymbol{x} - \boldsymbol{\mu}| \geq R - 1 - |\boldsymbol{\mu}|.
\]
Thus
\[
\exp \left( -\frac{1}{2}|\boldsymbol{x} - \boldsymbol{\mu}|^2 \right) \leq \exp \left( -\frac{1}{2}(R - 1 - |\boldsymbol{\mu}|)^2 \right).
\]
Therefore,
\[
\int_{\mathbb{R}^d \setminus H} |\boldsymbol{x}|^4 \exp\{2B|\boldsymbol{x}|\} q(\boldsymbol{x}) \mathrm{d}\boldsymbol{x} \leq (2\pi)^{-d/2} \exp \left( -\frac{1}{2}|\boldsymbol{\mu}|^2 \right) \int_{\mathbb{R}^d \setminus B(0, R-1)} |\boldsymbol{x}|^4 \exp \left( -\frac{1}{4}|\boldsymbol{x}|^2 \right) \mathrm{d}\boldsymbol{x}.
\]
Evaluating the integral in spherical coordinates, we obtain
\[
\int_{\mathbb{R}^d \setminus B(0, R-1)} |\boldsymbol{x}|^4 \exp \left( -\frac{1}{4}|\boldsymbol{x}|^2 \right) \mathrm{d}\boldsymbol{x} \leq \frac{(4 + 2)!!}{4\sqrt{\pi}} R^3 e^{-\frac{R^2}{4}}.
\]
Thus
\[
\int_{\mathbb{R}^d \setminus H} |\boldsymbol{x}|^4 \exp\{2B|\boldsymbol{x}|\} q(\boldsymbol{x}) \mathrm{d}\boldsymbol{x} \leq C \exp \left( -\frac{R^2}{4} \right),
\]
which satisfies the condition for sufficiently large $R$.

 Furthermore, if \(\frac{f}{q} \in \mathcal{S}(\mathbb{R}^d) \) and satisfies the growth condition \cref{term:GrowthCondition_Extended}, then by Theorem \ref{thm:MainTheorem}, we can get \cref{eq:ODIS_error}.
\end{proof}
% \begin{proof}
%     See the proof of Appendix \ref{proof_ODIS}. 
% \end{proof}
\section{The Applications of RQMC-IS in Finance Via Preintegration}\label{sec:Preintegration}

 In this section, we focus on discontinuous integrands of the form
\begin{equation}\label{eq:f_form}
 f(\x)=g(\x )\mathbb{I} \{ \phi (\x )\geq 0\} ,
\end{equation}
where $\x \in \mathbb{R}^d$, $g$ and $\phi$ are smooth functions. In option pricing, many payoff funtions can be written in this form. 
We consider calculating the expectation 
\begin{equation*}
    \mathbb{E}[f(\boldsymbol{X})] = \int_{\mathbb{R}^d} f(\x) \prod_{j=1}^d \rho_j(x_j) d\x,
\end{equation*}
where $\rho_j$ is an one-dimensional probability density function (p.d.f.).

Since $f$ is not a smooth function, to apply the results from the previous sections on RQMC-IS, we need to smooth the integrand $f$. To achieve this, we introduce the preintegration approach.

For fixed $j \in 1{:}d$, denote $\x_{-j}$ as the $d-1$ components of $\x$ apart from $x_j$. Integrating $f(\x)$ in \eqref{eq:f_form} with
respect to $x_j$ (i.e., taking $\x_{-j}$ as the conditioning variables) gives
\begin{equation*}\label{eq:condint}
(P_j f)(\x_{-j}):=\E[f(\x)|\x_{-j}]=\int_{-\infty}^{\infty} f(x_j,\x_{-j})\rho_j(x_j) dx_j,
\end{equation*}
where $P_j f$ denotes the function obtained by taking the preintegration of $f$ with respect to $x_j$.

Next, we study the smoothness and growth rates of the function $P_j f$. To derive the desired results, we follow the assumptions and use the conclusions established in \cite{preintegration_kuo2018, he2019}.
Firstly,we need $\phi (\x ) $ in \cref{eq:f_form} to be strictly monotone with respect
 to $x_j$.
\begin{assumption}\label{assump_mono}
    Let $j\in 1{:}d$ be fixed. Assume that 
    \begin{equation*}\label{eq:assump1}
    (\partial_j\phi)(\x)\neq 0 \text{ for all } \x\in \mathbb{R}^d.
    \end{equation*}
\end{assumption}
Under this assumptions, the results in \cite{he2019} provide conclusions about both the smoothness of $P_jf$ and its growth rate.

\begin{theorem}\label{thm:grie}
    Let $r$ be a positive integer. Suppose that $f$ is given by \eqref{eq:f_form} with $g,\phi\in  \mathcal{C}^r(\R^d)$ and $\E{\abs{f(\bm X)}}<\infty$, $\rho\in \mathcal{C}^{r-1}(\R)$, and Assumption~\ref{assump_mono} is satisfied. Denote $\y=\x_{-j}$. Let
    \begin{align*}
    U_j &= \{\y\in \R^{d-1}|\phi(x_j,\y)=0 \text{ for some }x_j\in \R\},\label{eq:uj}\\
    U_j^+ &= \{\y\in \R^{d-1}|\phi(x_j,\y)>0 \text{ for all  }x_j\in \R\},\notag\\
    U_j^- &= \{\y\in \R^{d-1}|\phi(x_j,\y)<0 \text{ for all  }x_j\in \R\}.\notag
    \end{align*}
    Then $U_j$ is open, and there exists a unique function $\psi\in \mathcal{C}^r(U_j)$ such that 
    $\phi(\psi(\y), \y) = 0$ for all $\y\in U_j$. 
    
    Assume that the following two condition is satisfied.
    \begin{enumerate}[label=(\Alph*)]
        \item  For any $\y^*\in U_j$, there exists a ball $B(\y^*, \delta) := \left\{ \y : |\y - \y^*| < \delta \right\}
\subseteq U_j$ with $\delta>0$ such that  $\int_{-\infty}^\infty \partial^{\bm \alpha} g(
    x_j,\y)\rho(x_j)d x_j$ converges uniformly on $B(\y^*,\delta)$ for any multi-index $\bm{\alpha}$ satisfying $\abs{\bm{\alpha}}\le r$ and $\alpha_j=0$.
        \item Every function over $U_j$ of the form
    \begin{equation}\label{eq:hy}
    h(\y) = \beta\frac{(\partial^{\bm \alpha^{(0)}}g)(\psi(\y),\y)\prod_{i=1}^a(\partial^{\bm \alpha^{(i)}}\phi)(\psi(\y),\y)}{[(\partial_j \phi)(\psi(\y),\y)]^b}\rho^{(c)}(\psi(\y)),
    \end{equation}
    where $\beta$ is a constant, $a,b,c$ are integers, and $\bm\alpha^{(i)}$ are multi-indices with the constraints $1\le a\leq 2r-1$, $1\le b\le 2r-1$, $0\le c\le r-1$, $\abs{\bm\alpha^{(i)}}\le r$, satisfies 
    \begin{equation}\label{eq:bounds}
    h(\bm y)\to 0\text{ as }\bm y\text{ approaches a boundary point of }U_j\text{ lying in }U_j^+\text{ or }U_j^-.
    \end{equation}
    \end{enumerate}

    Then $P_j f\in \mathcal{C}^r(\R^{d-1})$, and for every multi-index $\bm \alpha$ with $\abs{\bm \alpha}\leq r$ and $\alpha_j=0$,
    \begin{equation}\label{eq:bound4dfj}
    \abs{(\partial^{\bm \alpha} P_j f)(\bm y)}\le\int_{-\infty}^{\infty}\abs{(\partial^{\bm \alpha} g)(x_j,\bm y)}\rho(x_j)d x_j+\sum_{i=1}^{M_{\abs{\bm \alpha}}} \abs{h_{\bm \alpha,i}(\bm y)},
    \end{equation}
    where $M_{\abs{\bm \alpha}}$ is a nonnegative integer depending only on $\abs{\bm \alpha}$, and for $\bm y\in U_j$ and $\abs{\bm{\alpha}}>0$, $h_{\bm \alpha,i}(\bm y)$ has the form \eqref{eq:hy} with parameters satisfying $1\le a\leq 2\abs{\bm \alpha}-1$, $1\le b\le 2\abs{\bm \alpha}-1$, $0\le c\le \abs{\bm \alpha}-1$, $|\bm\alpha^{(i)}|\le \abs{\bm \alpha}$, otherwise $h_{\bm \alpha,i}(\y)=0$.
\end{theorem}
% \begin{proof}
%     See the proof of  Theorem 3.2 and 3.4 in \cite{he2019}. 
% \end{proof}
\begin{assumption}\label{assum:grc}
    Suppose that the integers $j\in 1{:}d$ and $r\ge d$ are fixed. There exist constants $L,A>0$  such that 
    % \begin{align}
    % \abs{(D^{\bm \alpha}q)(\x)}&\leq Le^{A\abs{\x}},\label{eq:dgup}\\
    % \abs{(D^{\bm \alpha}\phi)(\x)}&\leq Le^{A\abs{\x}},\label{eq:dphi}\\	
    % \abs{(D_j\phi)(\x)}^{-b}&\leq Le^{A\abs{\x}}\label{eq:dphib}
    % \end{align}
        \begin{equation}\label{eq:gc_q_phi}
            \max \{\abs{(\partial^{\bm \alpha}g)(\x)},\abs{(\partial^{\bm \alpha}\phi)(\x)},\abs{(\partial_j\phi)(\x)}^{-b}  \}\leq  Le^{A\abs{\x}},
        \end{equation}
    hold for $1\le b\le 2r-1$ and any multi-index $\bm \alpha$ satisfying $\abs{\bm \alpha}\le r$.	
\end{assumption}
\begin{remark}
This assumption is quite natural in financial engineering. Under the Black–Scholes model (see next section), option payoff functions can typically be written in the form given in \eqref{eq:f_form}, where functions $g$ and $\phi$ can be written as functions of the form $e^{\bm A\x}$, with the matrix $ \bm A \in \mathbb{R}^{d \times d}$ . As a result, condition \cref{eq:gc_q_phi} is usually easy to satisfy. Furthermore, the growth condition \cref{eq:gc_q_phi} is relatively simple to verify, especially when compared with Assumption 4.5 in \cite{he2019}, which is often difficult to verify in practice.
\end{remark}
In this paper, we consider the case where the density is the standard normal density, i.e., $\rho = \varphi$. Based on the growth condition \cref{eq:gc_q_phi}, we can establish the following theorem, which provides an upper bound on the growth of the function after preintegration with respect to the $j$th component.
\begin{theorem}\label{thm:pvf}
Suppose that $f$ is given by \eqref{eq:f_form} with $g,\phi\in  \mathcal{C}^r(\R^d)$, where $r$ is a positive integer. 
Suppose that Assumptions~\ref{assump_mono} and \ref{assum:grc} are satisfied with $r$ and fixed constants $L, A >0$, ~and~ $j\in 1{:}d$. Then there exist constants $L', A'>0$ such that
\begin{equation}
    \abs{(\partial^{\bm \alpha} P_j f)(\bm y)}=\int_{-\infty}^{\infty} f(x_j,\y)\varphi(x_j) dx_j\le L'e^{A'\abs{\boldsymbol{y}}},
\end{equation}\label{eq:gc_pf}
where $\varphi(x)$ is the one-dimensional standard normal density function.
\end{theorem}
\begin{proof}
Firstly, Assumption~\ref{assum:grc} ensures that $\E{\abs{f(\bm X)}}<\infty$.
In what follows, we just need to prove that the conditions (A) and (B) in Theorem~\ref{thm:grie} are satisfied, then we can get the existence and growth condition of $(\partial^{\bm \alpha} P_j f)(\bm y)$.

Now we verify the condition (A). Let \( \boldsymbol{y}^* \) be an arbitrary point in \( \mathbb{R}^{d-1} \). For any \( \delta > 0 \), from \eqref{eq:gc_q_phi}, we have
\[
\sup_{\boldsymbol{y} \in B(\boldsymbol{y}^*, \delta)} \left| \partial^{\bm \alpha} g(x_j, \boldsymbol{y}) \varphi(x_j) \right| \leq \sup_{\boldsymbol{y} \in B(\boldsymbol{y}^*, \delta)} (Le^{A(\abs{x_j}+\abs{\y})}\varphi(x_j))
\leq M(\delta)e^{A\abs{x_j}}\varphi(x_j),
\]
where \( M(\delta) \) is a constant depending on \( \delta \). Together with
\[
\int_{-\infty}^{\infty}  e^{A\abs{x_j}}\varphi(x_j) \, dx_j < \infty,
\]
Weierstrass test (see Theorem 3.5 in \cite{he2019}) admits that \( \int_{-\infty}^{\infty} \partial^{\alpha} g(x_j, \boldsymbol{y}) \varphi(x_j) \, dx_j \) converges uniformly on \( B(\boldsymbol{y}, \delta) \). So condition (A) in Theorem~\ref{thm:grie} is satisfied.

We next verify condition (B). For the function \( h \) given in \eqref{eq:hy}, by Assumption~\ref{assum:grc}, we find that
\begin{equation*}
|h(\boldsymbol{y})| \leq \kappa(\psi(\boldsymbol{y})) e^{(2r+1)A\abs{\y}},
\end{equation*}
where
\[
\kappa(x) := |\beta| L^{2r+1}e^{(2r+1)A\abs{x}} \varphi^{(c)}(x).
\]
Then \( \kappa(x) \to 0 \) as \( x \to \pm \infty \). To verify condition (B), it suffices to show that \( \psi(\boldsymbol{y}) \) goes to infinity as \( \boldsymbol{y} \) approaches a boundary point of \( U_j \) lying in \( U_j^- \) and \( U_j^+ \). Without loss of generality, we suppose that \( (\partial_j \phi)(x) > 0 \) in Assumption~\ref{assump_mono}. This implies that \( \phi(x_j, \boldsymbol{y}) \) is an increasing function with respect to \( x_j \) for given \( \boldsymbol{y} \). We then have
\[
U_j^- = \left\{ \boldsymbol{y} \in \mathbb{R}^{d-1} \middle| \lim_{x_j \to \infty} \phi(x_j, \boldsymbol{y}) \leq 0 \right\}.
\]

As a result, \( \psi(\boldsymbol{y}) \to \infty \) as \( \boldsymbol{y} \) approaches a boundary point of \( U_j \) lying in \( U_j^- \). Also,
\[
U_j^+ = \left\{ \boldsymbol{y} \in \mathbb{R}^{d-1} \middle| \lim_{x_j \to -\infty} \phi(x_j, \boldsymbol{y}) \geq 0 \right\}.
\]

Similarly, \( \psi(\boldsymbol{y}) \to \infty \) as \( \boldsymbol{y} \) approaches a boundary point of \( U_j \) lying in \( U_j^+ \). Theorem~\ref{thm:grie} gives \( P_j f \in C^r(\mathbb{R}^{d-1}) \).

Finally, by Theorem~\ref{thm:grie}, we have
\begin{equation*}
\abs{(\partial^{\bm \alpha} P_j f)(\bm y)}\le\int_{-\infty}^{\infty}\abs{(\partial^{\bm \alpha} g)(x_j,\bm y)}\varphi(x_j)d x_j+\sum_{i=1}^{M_{\abs{\bm \alpha}}} \abs{h_{\bm \alpha,i}(\bm y)} \le L'e^{A'\abs{\y}}.
\end{equation*}

This completes the proof.
\end{proof}
After imposing Assumption \ref{assump_mono} and Assumption \ref{assum:grc} on the discontinuous functions in finance of the form \cref{eq:f_form}, we derived an upper bound for the growth of the preintegrated function. Consequently, we can apply the RQMC-ODIS convergence theory obtained in the previous section to arrive at the following corollary.
\begin{corollary}\label{cor}
Under the same assumptions in Theorem~\ref{thm:pvf}, we have
\begin{equation*}
\left(\mathbb{E} \left[ \left| \hat{I}_n ((P_j f)\varphi/q) - \mathbb{E}[f(\boldsymbol{X})]\right|^2 \right] \right)^{1/2} = O(n^{-1  + \epsilon}),
\end{equation*}
where $q$ is the p.d.f. of the ODIS-proposal \cref{ODIS_proposal} and $\varphi$ is a $(d-1)$-dimension standard normal density function.
\end{corollary}
\begin{proof}
    By Leibniz rule, we have
\begin{equation*}
\left| \partial^{\boldsymbol{u}} (P_jf/q) \right| = \left| \sum_{\boldsymbol{u}_1 + \boldsymbol{u}_2 = \boldsymbol{u}} \partial^{\boldsymbol{u}_1} P_jf~ \partial^{\boldsymbol{u}_2} (1/q) \right| 
\le\sum_{\boldsymbol{u}_1 + \boldsymbol{u}_2 = \boldsymbol{u}} \abs{\partial^{\boldsymbol{u}_1} P_jf}~ \abs{\partial^{\boldsymbol{u}_2} (1/q)}.
\end{equation*}
Then combined with Theorem \ref{thm:MainTheorem} and Theorem~\ref{thm:pvf}, we can get the result.\end{proof}

\section{Numerical Experiments}\label{sec:NumericalExperiments}
In this section, we perform some numerical experiments on pricing financial options whose payoff functions have the form~\eqref{eq:f_form} under either the Black-Scholes model or
 the Heston model. In our experiments, RQMC is implemented as a linear-scrambled version of
 Sobol's points proposed in \cite{linear_scramble} and the ODIS is used in all the following numerical experiments. In all of the following experiments, we will focus on comparing the Root Mean Square Error (RMSE) of the four methods — plain MC, plain QMC, QMC+preintegration or conditional QMC (CQMC), and CQMC+IS.
\subsection{Black-Scholes Model}
Let $S_t$ denote the underlying price dynamics at time $t$ under the risk-neutral measure. Under the Black-Scholes model, $S_t$ follows a geometric Brownian motion satisfying the stochastic differential equation
\begin{equation}\label{eq:BM_SDE}
dS_t =rS_tdt+\sigma S_tdB_t,
\end{equation}
where $r$ is
the risk-free interest rate, $\sigma$ is the volatility and $B_t$
is the standard Brownian motion. The solution of \eqref{eq:BM_SDE} is analytically available
\begin{equation}
S_t=S_0\exp[(r-\sigma^2/2)t+\sigma B_t],\label{Solution}
\end{equation}
where $S_0$ is the initial price of the asset.
Suppose that the maturity time of the option is $T$. In practice, we simulate a discrete Brownian motion. Without loss of generality,  we assume that the discrete time step satisfies $t_i=i\Delta t$, where $\Delta t=T/d$. Denote $S_i=S_{t_i}$. 
Let $\bm{B}:=(B_{t_1},\dots,B_{t_d})^{\top}$. We have $\bm{B}\sim N (\bm{0},\bm{\Sigma})$, where  $\bm{\Sigma}$ is a positive definite matrix with entries $\Sigma_{ij}=\Delta t\min(i,j)$.

Let $\bm{A} := (a_{ij})_{d \times d}$ be a matrix satisfying $\bm{A}\bm{A}^\top=\bm{\Sigma}$. Using the transformation $\bm{B}=\bm{A}\bm x$, where $\x :=(x_{1},\dots,x_d)^{\top} \sim N(\bm 0,\bm I_d)$, it follows from
\eqref{Solution} that
\begin{align}\label{eq:si}
S_i= S_0\exp\left[(r-\sigma^2/2)i\Delta t+\sigma \sum_{j=1}^da_{ij}x_j\right].
\end{align}

\subsubsection{Single-asset Case}
In the single-asset case, the first step is to choose a generating matrix $\bm A$. There are several methods for selecting $A$, with Cholesky decomposition being the most commonly used. Three alternative approaches, Brownian bridge construction \cite{Caflisch1997, BB}, principal component analysis (PCA) \cite{PCA} and gradient PCA (GPCA) 
 \cite{GPCA}, are often used in combination with QMC methods to reduce the effective dimension.

We consider the arithmetic Asian option, whose discounted payoff function is given by
\begin{equation}\label{eq:oneassetpf}
e^{-rT}(S_A - K)_+ = e^{-rT} \max(S_A - K, 0),
\end{equation}
where $K$ is the strike price, and $S_A = (1/d)\sum_{j=1}^d S_j$. 
Regardless of the matrix generation method chosen, it can be verified that Assumption \ref{assump_mono} is satisfied \cite{he2019},
and obviously, the function given by ~\eqref{eq:oneassetpf} satisfies Assumption \ref{assum:grc}.
Next, following \cite{zcj2020}, we provide the analytical expression for the preintegration of the payoff function with respect to $x_1$.

Denote $\omega = r - \sigma^2/2$ , $\boldsymbol{z} = (x_2, \ldots, x_d)^{\mathrm{T}}$
and $\boldsymbol{C}$ is a $(d - 1) \times (d - 1)$ matrix satisfying $\boldsymbol{C}\boldsymbol{C}^T = \hat{\boldsymbol{\Sigma}}$ with
$$
\hat{\boldsymbol{\Sigma}} := \begin{pmatrix}
t_2 - t_1 & t_2 - t_1 & \dots & t_2 - t_1 \\
t_2 - t_1 & t_3 - t_1 & \dots & t_3 - t_1 \\
\vdots & \vdots & \ddots & \vdots \\
t_2 - t_1 & t_3 - t_1 & \dots & t_d - t_1
\end{pmatrix}.
$$
Then we have
\begin{equation*}
\E[e^{-rT} (S_A - K)^+|\x_{-1}] 
= \mathrm{e}^{r(t_1 - T)} \widetilde{S}_A \left[ 1 - \Phi\left( \psi_d - \sigma \sqrt{t_1} \right) \right] - \mathrm{e}^{-rT} K \left[ 1 - \Phi\left( \psi_d \right) \right],
\end{equation*}
where
\begin{equation*}
\widetilde{S}_A = \frac{1}{d} \sum_{j=1}^d \widetilde{S}_j = \frac{1}{d} \sum_{j=1}^d S_0 \exp(\omega (t_j - t_1) + \sigma \boldsymbol{C}_{j-1} \boldsymbol{z}) ,
\psi_d = \frac{\ln K - \ln \widetilde{S}_A - \omega t_1}{\sigma \sqrt{t_1}}
\end{equation*}
and 
\begin{equation*}
\widetilde{S}_j = S_0 \exp(\omega (t_j - t_1) + \sigma \boldsymbol{C}_{j - 1}\boldsymbol{z}).
\end{equation*}

\begin{figure}
    \centering 
    \includegraphics[width=1.\textwidth]{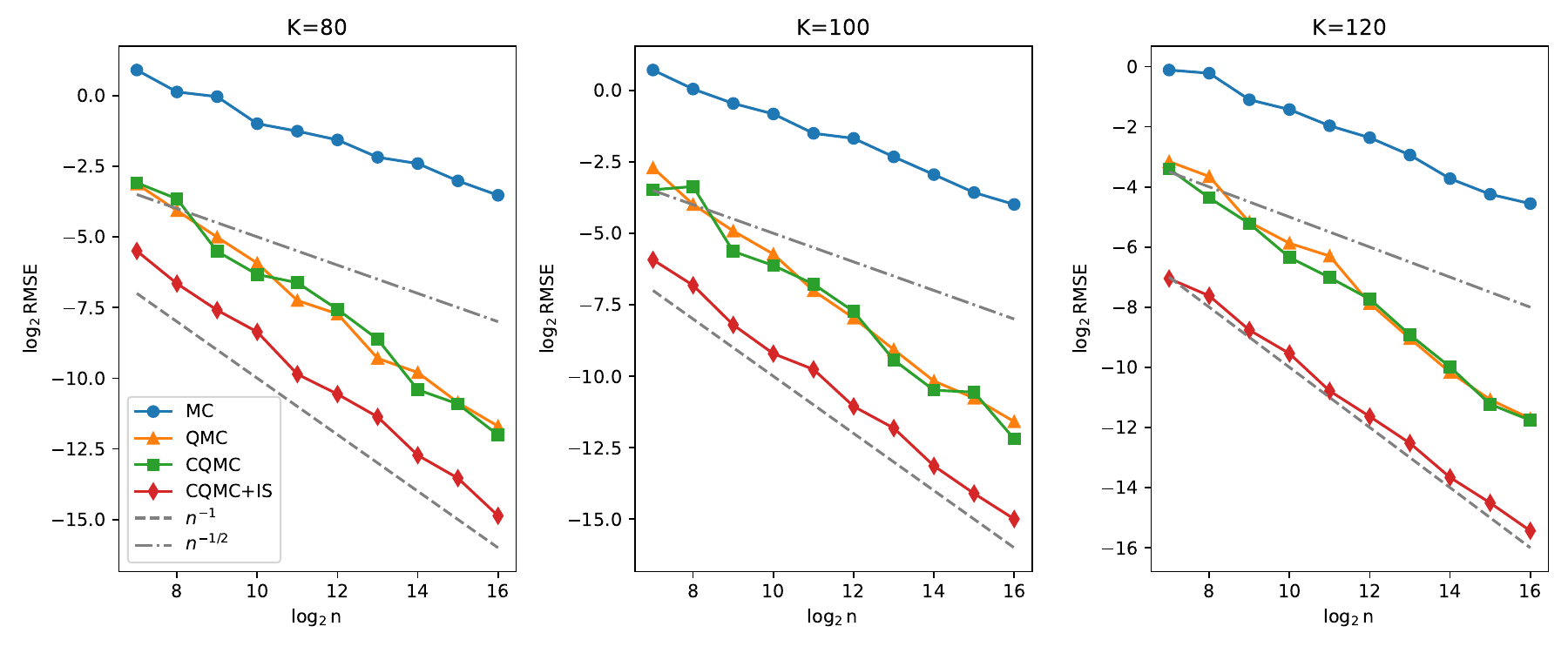}
    \caption{RMSEs for the single asset Asian option with $d=80$. The RMSEs are computed based on $100$ repetitions. The figure has two reference lines proportional to labeled powers of $n$.} 
    \label{Asian_d=80}
\end{figure} 

In the numerical experiments for the single-asset arithmetic Asian option, we set the parameters as follows: initial asset price $S_0 = 100$, volatility $\sigma = 0.4$, maturity $T = 1$, dimension $d = 80$, risk-free rate $r = 0.1$, and strike prices $K \in \{80, 100, 120\}$. RMSEs of RQMC are estimated based on 100 repetitions. We use the CQMC+IS method with a large sample size $n = 2^{19}$ to obtain a high-precision estimate, which is treated as the true value. GPCA is used as the path generation method throughout all the situations. We choose the sample sizes $n = 2^7,\dots,2^{16}$. Figure \ref{Asian_d=80} shows the RMSEs of plain MC, plain QMC, CQMC, and CQMC+IS methods under different strike prices $K = 80, 100, 120$ for $d = 80$, respectively.

It can be observed that the convergence rate of the CQMC+IS method closely matches the theoretical rate of $n^{-1}$.
CQMC+IS achieves the lowest RMSE, leveraging preintegration and importance sampling to reduce variance. For $d = 80$, QMC and CQMC perform better than MC, but they are outperformed by CQMC+IS, which effectively improves efficiency through conditional smoothing and dimension reduction.  
\subsubsection{Multi-assets Case}
In this subsection, we study a multi-assets option, which is called the basket option. A basket option depends on a weighted average of the prices of several assets \cite{lsf2023}. Suppose that under the risk-neutral measure, the $L$ assets $S^{(1)},\ldots,S^{(L)}$ follow
\begin{equation*} 
    \mathrm{d}S_t^{(l)} = r S_t^{(l)} \, \mathrm{d}t + \sigma_l S_t^{(l)} \, \mathrm{d}B_t^{(l)},
\end{equation*}
where $\{B^{(l)}_t\}_{1 \leq l \leq L}$ are standard Brownian motions with correlation $\rho$, i.e.,  
    $Cov(B^{(l)}_s, B^{(k)}_t) = \rho \min(s, t)$
for all $t,s > 0$ and $l,k~\in \{1, 2 \dots L\}, l \neq k$. For some nonnegative weights $w_1 + \cdots + w_L = 1$, the payoff function of the basket Asian call option is given by
\begin{equation*}
    \left( \sum_{l=1}^L w_l \bar{S}^{(l)} - K \right)_+,
\end{equation*}
where $\bar{S}^{(l)} = \frac{1}{d} \sum_{j = 1}^d S_j^{(l)}$ is the arithmetic average of $S_t^{(l)}$ in the time interval $[0, T]$.
% \begin{equation*}
%     \bar{S}^{(l)} = \frac{1}{d} \sum_{j = 1}^d S_j^{(l)}.
% \end{equation*}
Here, we only consider $L = 2$ assets. To generate $B^{(1)}, B^{(2)}$ with correlation $\rho$, we can generate two independent standard Brownian motions $W^{(1)}, W^{(2)}$ and let $B^{(1)} = W^{(1)}$, $B^{(2)} = \rho W^{(1)} + \sqrt{1 - \rho^2} W^{(2)}$. Following the same discretization as before, we can generate $ \boldsymbol{z} \sim N (0, I_{2d})$. Then for time steps $j = 1, \ldots, d$, let
\begin{equation*}
    S_j^{(1)} = S_0^{(1)} \exp\left( \left( r - \frac{\sigma_1^2}{2} \right) j \Delta t + \sigma_1 \left( \rho \sum_{k=1}^d \bm A_{jk} z_{k+d} + \sqrt{1 - \rho^2} \sum_{k=1}^d \bm A_{jk} z_k \right) \right)
\end{equation*}
and
\begin{equation*}
    S_j^{(2)} = S_0^{(2)} \exp\left( \left( r - \frac{\sigma_2^2}{2} \right) j \Delta t + \sigma_2 \sum_{k=1}^d \bm A_{jk} z_{k+d}\right).
\end{equation*}
Then the discounted payoff function becomes
\begin{equation}\label{eq:multiasset_pf}
      e^{-rT}\left( \bar{S}- K \right)_+,
\end{equation}
where
\begin{equation*}
    \bar{S} = \frac{w_1}{d}\sum_{j=1}^d S_j^{(1)}+\frac{w_2}{d}\sum_{j=1}^d S_j^{(2)}.
\end{equation*}
Obviously, the discounted payoff funtion of the option satisfies the Assumption \ref{assum:grc}. We choose the cholesky decomposition to be the matrix $\bm A$. Thus all coefficients preceding $z_1$ are positive, and Assumption \ref{assump_mono} is thereby naturally satisfied. Next, we derive the preintegration expression of equation \eqref{eq:multiasset_pf} with respect to $z_1$. Due to Assumption \ref{assump_mono} is satisfied, there exists a function $\psi_d$ such that $\{\bar{S} > K\} = \{\bar{S}(\psi_d(\z_{-1}), \z_{-1}) > K\} = \{z_1 > \psi_d\}$.
Then we have
\begin{align*}
    &\E[e^{-rT} (\bar{S}- K)_+|\z_{-1}] 
    =\int_{-\infty}^{+\infty} e^{-rT}(\bar{S}-K)\mathbb{I}\{\bar{S}-K>0\}\varphi(z_1)dz_1\\   
    =&\int_{-\infty}^{+\infty} e^{-rT}\left(\frac{w_1}{d}\sum_{j=1}^d S_j^{(1)}+\frac{w_2}{d}\sum_{j=1}^d S_j^{(2)}-K\right)\mathbb{I}\{z_1>\psi_d\}\varphi(z_1)dz_1\\
    =&\frac{w_1}{d}\sum_{j=1}^d \exp\left(-rT+\frac{1}{2}\sigma_1^2(1-\rho^2)A_{j1}^2+\left(r-\frac{\sigma_1^2}{2}\right)j\Delta t
    +\sigma_1 \left( \rho \sum_{k=1}^d A_{jk} z_{k+d} + \sqrt{1 - \rho^2} \sum_{k=1}^d A_{jk} z_k \right)\right)\\
    \cdot&(1-\Phi(\psi_d-\sigma_1\sqrt{1-\rho^2}A_{j1}))+e^{-rT}\left(\sum_{j=1}^d S_j^{(2)}-K\right)(1-\Phi(\psi_d)).
\end{align*}
\begin{figure}
    \centering 
    \includegraphics[width=1.\textwidth]{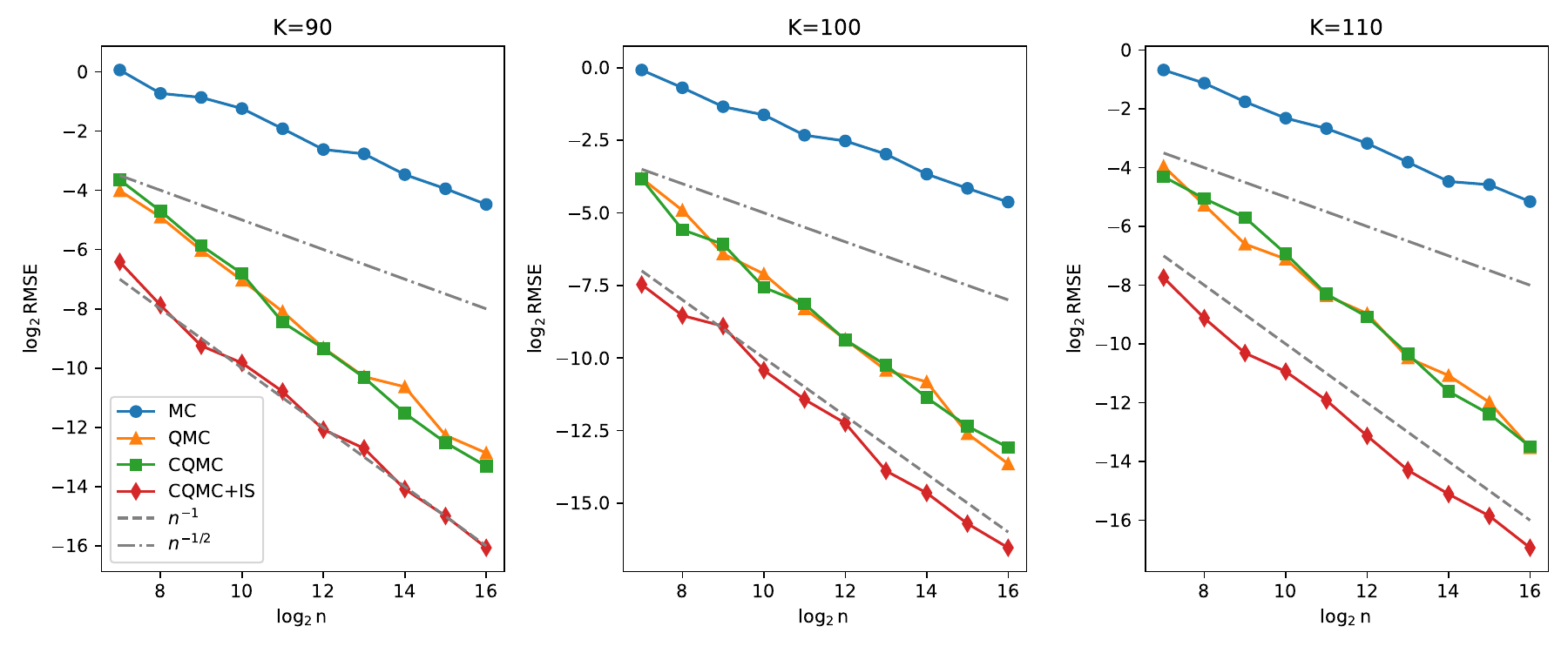}
    \caption{RMSEs for the basket Asian option with $2d=30$. The RMSEs are computed based on $100$ repetitions. The figure has two reference lines proportional to labeled powers of $n$.} 
    \label{2asset_d=15}
\end{figure}
\begin{figure}
    \centering 
    \includegraphics[width=1.\textwidth]{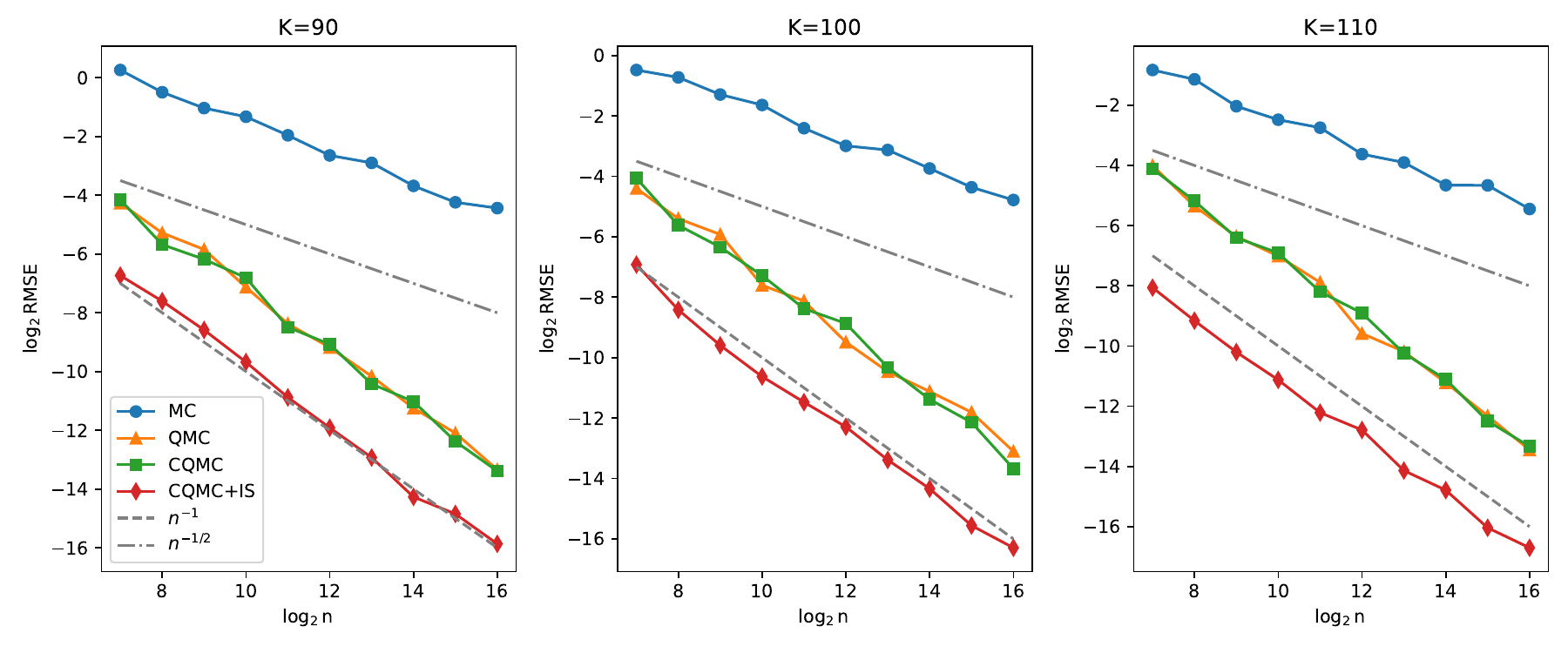}
    \caption{RMSEs for the basket Asian option with $2d=60$. The RMSEs are computed based on $100$ repetitions. The figure has two reference lines proportional to labeled powers of $n$.} 
    \label{2asset_d=30}
\end{figure}

The parameters are set as follows: initial prices of the two assets $S_0^{(1)} = S_0^{(2)} = 100$, volatilities $\sigma_1 = 0.3$, $\sigma_2 = 0.1$, maturity $T = 1$, weights $w_1 = 0.7$, $w_2 = 0.3$, dimensions $d =  15, 30$, risk-free rate $r = 0.05$, strike prices $K \in \{90, 100, 110\}$, and correlation coefficient $\rho = 0.5$. The nominal dimensions are $2d = 30, 60$. The RMSEs of RQMC are estimated based on 100 repetitions. We use the CQMC+IS method with a large sample size $n = 2^{19}$ to obtain an accurate estimation and treat it as the true value. GPCA is used as the path generation method throughout all the situations. We choose the sample sizes $n = 2^7,\dots,2^{16}$. Figures \ref{2asset_d=15} and \ref{2asset_d=30} show  the RMSE of plain MC, plain QMC, CQMC, and CQMC+IS methods under different strike prices $K = 90, 100, 110$ for $d = 15$ and $d = 30$, respectively.

It can be observed that the convergence rate of the CQMC+IS method closely matches the theoretical rate of $n^{-1}$.
Notably, for out-of-the-money options (e.g., \(K = 110\)), the variance reduction effect of IS is more pronounced: CQMC+IS achieves a significantly lower RMSE compared to other methods. This is because out-of-the-money scenarios involve rarer payoff events, and IS effectively targets these events, enhancing the efficiency.  
The weighted average structure of the basket option introduces additional complexity, but CQMC+IS still maintains the lowest RMSE. As the dimension increases (from \(d = 15\) to \(d = 30\)), the performance gap between CQMC+IS and other methods widens: QMC and CQMC outperform MC, but are surpassed by CQMC+IS, which handles multi-asset correlation and high-dimensionality through importance sampling and conditional smoothing.
\subsection{Heston Model}
Under the Heston framework, the risk-neutral dynamics of the asset can be expressed as
\begin{equation*}
    dS_t = rS_t dt + \sqrt{V_t}S_t dB^{(1)}_t,
\end{equation*}
\begin{equation*}
    dV_t = (\theta - V_t)\nu dt + \sigma \sqrt{V_t} dB^{(2)}_t,
\end{equation*}
where $\nu$ is the mean-reversion parameter of the volatility process $V_t$, $\theta$ is the long run average price variance, $\sigma$ is the volatility of the volatility, $B^{(1)}_t$ and $B^{(2)}_t$ are two standard Brownian motions with an instantaneous correlation $\rho$, i.e., $Cov(B^{(1)}_s, B^{(2)}_t) = \rho \min(s, t)$
for any $s, t > 0$. One may write that $B^{(1)}_t = \hat{\rho} W^{(1)}_t + \rho W^{(2)}_t$ and $B^{(2)}_t = W^{(2)}_t$, where $\hat{\rho} = \sqrt{1 - \rho^2}$ and $W^{(1)}_t$ and $W^{(2)}_t$ are two independent standard Brownian motions.
Let $\boldsymbol{z} = (z_1, z_2, \ldots, z_{2d-1}, z_{2d})^\mathrm{T} \sim N(\boldsymbol{0}, \bm I_{2d})$. 
We use the Euler-Maruyama scheme to discretize the asset paths \cite{heston}, resulting in
\begin{equation*}
S_i = S_0 \exp\left\{ ri\Delta t - \Delta t \sum_{j=0}^{i-1} V_j/2 + \sqrt{\Delta t} \sum_{j=0}^{i-1} \sqrt{V_j}\bigl(\hat{\rho} z_{2j+1} + \rho z_{2j+2}\bigr) \right\}
\end{equation*}
and
\begin{equation*}
V_i = V_0 + \theta \nu i\Delta t - \nu \Delta t \sum_{j=0}^{i-1} V_j + \sigma \sqrt{\Delta t} \sum_{j=0}^{i-1} \sqrt{V_j} z_{2j+2},
\end{equation*}
where $V_0$ is the initial value of the volatility process,  $S_i$ and $V_i$ represent the approximations of $S_{i\Delta t}$ and $V_{i\Delta t}$ for $\Delta t = T/m$, respectively. Here we still consider estimating the price of an arithmetic average Asian call option 
\begin{equation*}
\mathbb{E}\bigl[e^{-rT}(S_A - K)_+\bigr], 
\end{equation*}
where $S_A = (1/m)\sum_{i=1}^m S_i$. 

Next, we provide the analytically expression for the preintegration of the payoff function with respect to $z_1$. Assumption \ref{assump_mono} is satisfied, because all coefficients preceding $z_1$ are positive. The derivation is presented as follows (see \cite{zcj2021})
\begin{equation*}
\mathbb{E}\bigl[(S_A - K)_+ \bigm\vert \boldsymbol{z}_{-1}\bigr]  
= \tilde{S}_A e^{(1 - \rho^2)V_0 \Delta t/2} \bigl[1 - \Phi\bigl(\psi_d(\boldsymbol{z}_{-1}) - \hat{\rho}\sqrt{V_0 \Delta t}\bigr)\bigr] - K \bigl[1 - \Phi\bigl(\psi_d(\boldsymbol{z}_{-1})\bigr)\bigr],
\end{equation*}
where 
\begin{equation*}
\psi_d(\boldsymbol{z}_{-1}) = \log(K/\tilde{S}_A)/\sqrt{(1 - \rho^2)V_0 \Delta t}
\end{equation*}
and
\begin{equation*}
\tilde{S}_A = (S_0/m)\sum_{i=1}^m  \exp\left\{ ri\Delta t - \Delta t \sum_{j=0}^{i-1} V_j/2 + \rho \sqrt{\Delta t V_0}\, z_2 + \sqrt{\Delta t} \sum_{j=1}^{i-1} \sqrt{V_j}\bigl(\hat{\rho} z_{2j+1} + \rho z_{2j+2}\bigr) \right\}.
\end{equation*}
\begin{remark}
If the growth condition \cref{eq:gc_q_phi} in Assumption \ref{assum:grc} is modified to be controlled by $Le^{A\abs{\x}^r}$  with $1\leq r<2$, then following the proof of Theorem \ref{thm:pvf} , it can be deduced that the preintegrated function is bounded by such an exponential function. In this case, the convergence order specified in Corollary \ref{cor} remains valid.  
Within the Heston model, the payoff function satisfies this modified Assumption \ref{assum:grc}. Consequently, the convergence order can still be attained.  
\end{remark}

\begin{figure}
    \centering 
    \includegraphics[width=1.\textwidth]{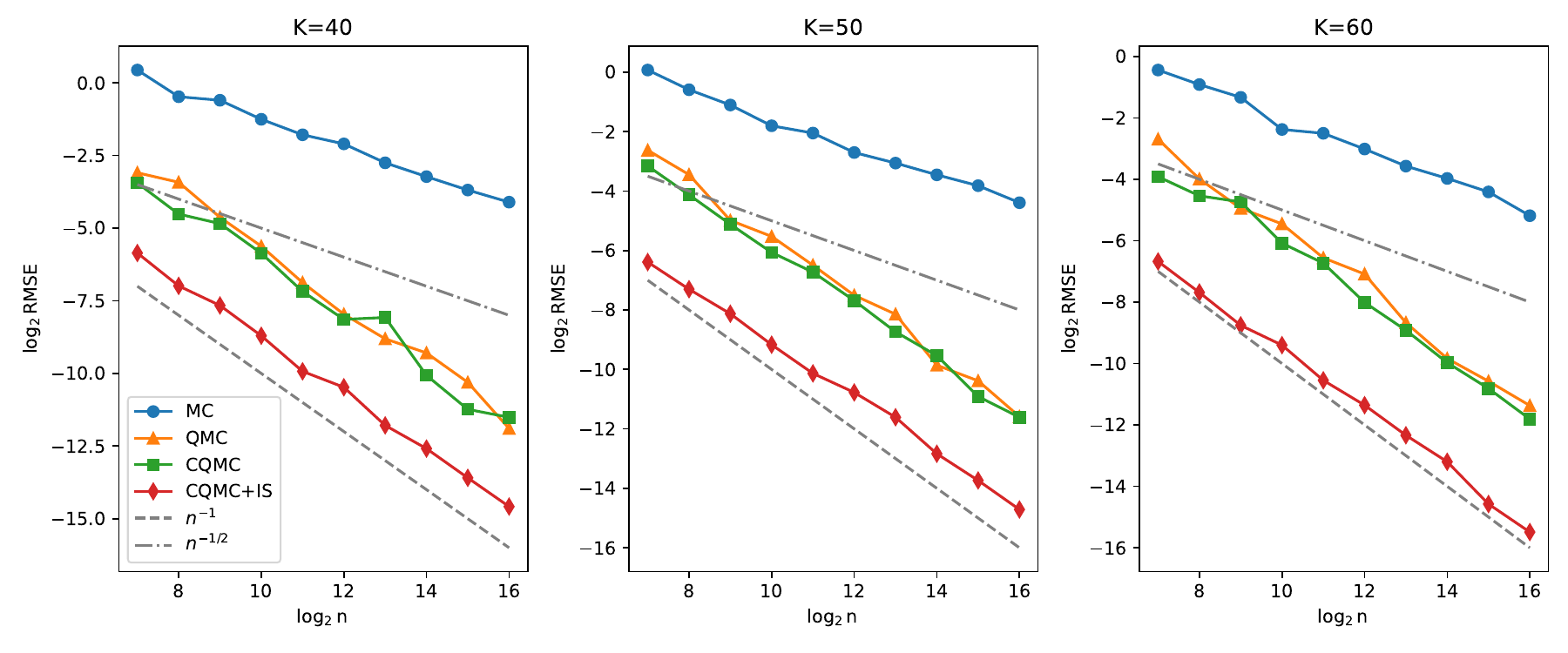}
    \caption{RMSEs for the Heston model with 2d=10. The RMSEs are computed based on 100 repetitions.The figure has two reference lines proportional to labeled powers of n.} 
    \label{heston_d=5}
\end{figure}
\begin{figure}
    \centering 
    \includegraphics[width=1.\textwidth]{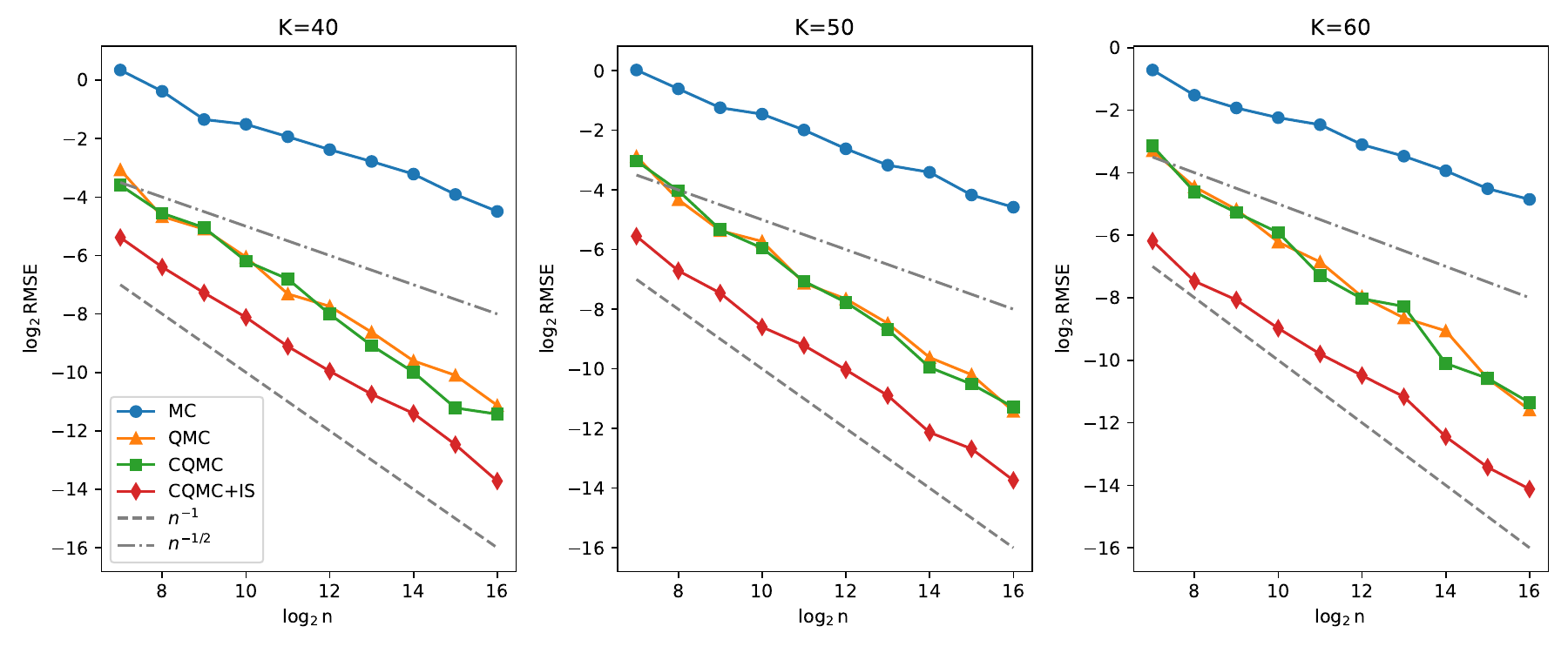}
    \caption{RMSEs for the Heston model with $2d=30$. The RMSEs are computed based on $100$ repetitions.The figure has two reference lines proportional to labeled powers of $n$.} 
    \label{heston_d=15}
\end{figure}
\begin{figure}
    \centering 
    \includegraphics[width=1.\textwidth]{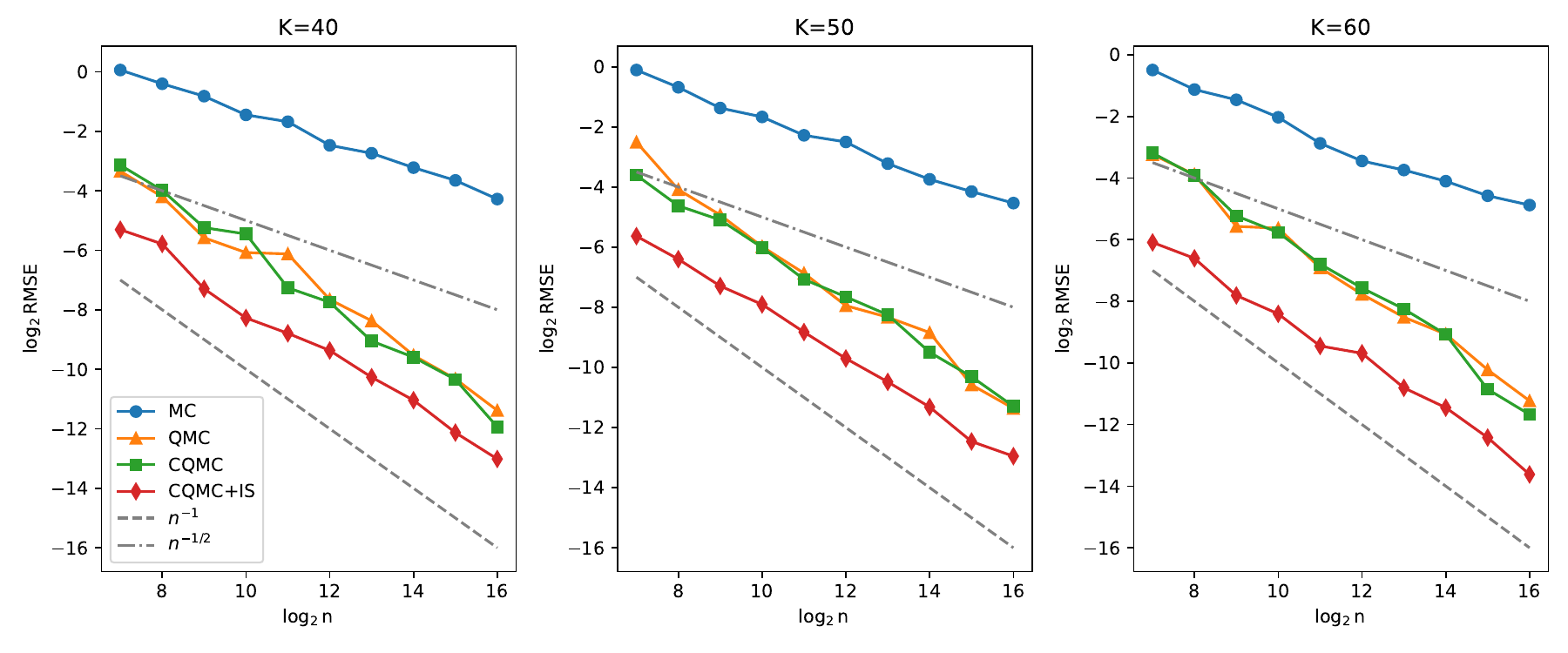}
    \caption{RMSEs for the Heston model with $2d=60$. The RMSEs are computed based on $100$ repetitions.The figure has two reference lines proportional to labeled powers of $n$.} 
    \label{heston_d=30}
\end{figure}

In the numerical experiments, we choose $S_0 = 50$, $V_0 = \theta = \sigma = 0.2$, $T = 1$, $\nu = 1$, $d = 5, 15,30$, $r = 0.05$, $K \in \{40, 50, 60\}$, and $\rho = 0.5$. The nominal dimension is $2d = 10,30,60$. The RMSEs of RQMC are estimated based on 100 repetitions. We use the CQMC+IS method with a large sample size $n = 2^{19}$ to obtain an accurate estimation and treat it as the true value. GPCA is used as the path generation method throughout all the situations. We choose the sample sizes $n = 2^7,\dots,2^{16}$. Figures \ref{heston_d=5}, \ref{heston_d=15} and \ref{heston_d=30} show  the RMSEs of plain MC, plain QMC, CQMC, and CQMC+IS methods under different strike prices $K = 40, 50, 60$ for $d = 5, d = 15$ and $d = 30$, respectively.

It can be observed that the convergence rate of the CQMC+IS method closely matches the theoretical rate of $n^{-1}$.
 Stochastic volatility in the Heston model increases path complexity, but CQMC+IS still achieves the lowest  RMSE across all dimensions, leveraging preintegration and importance sampling to handle volatility randomness.
For $d=5, 15, 30$ (nominal dimension $2d=10, 30, 60$), QMC and CQMC perform better than MC but are outperformed by CQMC+IS, which effectively reduces variance through conditional smoothing and importance sampling.

\section{Conclusion}\label{sec:Conclusion}

In this paper, we addressed a significant gap in the convergence theory of RQMC-IS, particularly for problems involving integrands with critical exponential growth. Previous theories require the growth rate  constant \(A\) in \(e^{A|\boldsymbol{x}|^2}\) to be strictly less than \(1/2\), a condition that is often violated in practice and in cases when effective variance reduction techniques like ODIS are used.

Our primary contribution is the development of a novel, more general convergence theorem for RQMC-IS. This theorem extends the analysis to the critical case where \(A=1/2\), establishing a near \(O(n^{-1})\) convergence rate under a verifiable light-tailed condition on the proposal distribution. We then proved that the Gaussian proposals used in ODIS satisfy this condition, thus providing the first rigorous convergence guarantees for the widely used RQMC-ODIS method in these challenging scenarios.

Furthermore, by integrating our new convergence theory with the preintegration technique, we demonstrated that even discontinuous functions, once smoothed, can be efficiently estimated within our framework. We showed that if the components of the original payoff function have linear exponential growth, the smoothed function preserves this property, allowing it to be seamlessly combined with an IS proposal that introduces the critical quadratic exponential growth.

Numerical experiments on a variety of option pricing problems, including single-asset Asian options, basket options under Black-Scholes model and stochastic volatility (Heston) models, empirically validated our theoretical findings. The results consistently demonstrated that the CQMC+IS method not only achieves the predicted convergence rate but also significantly outperforms standard MC, QMC, and CQMC methods in terms of accuracy and efficiency.

Future research is expected to extend this work in several directions. One promising avenue is to investigate other families of the proposal distributions beyond the Gaussian family, such as those with heavier tails, and analyze their compatibility with our theoretical framework. Another direction is to explore more applications of this methodology to other high-dimensional integration problems beyond finance, for example, in Bayesian statistics or computational physics.

\appendix
\section{Appendix}\label{sec:Appendix}

This appendix provides detailed proofs for the key lemmas presented in the main text. We will prove Lemma \ref{lemma:ProjectionError} by the mean value theorem, and calculate the \(V_{HK}\) as in \cite{Ouyang2024} to prove the Lemma \ref{lemma:V_HK}. They both need the following lemma about \(\partial^{\boldsymbol{u}}(f_{IS})\). Thus we begin by establishing a bound on the derivatives of the IS-based integrand \( f_{IS} \).

\begin{lemma}\label{lemma:DerivativeOfh_IS}
If \( f/q \) as a function in \( \mathcal{S}(\mathbb{R}^d) \) satisfies the growth condition \cref{term:GrowthCondition_Extended} and \( \boldsymbol{u} \subseteq 1:d \), then
\[
|\partial^{\boldsymbol{u}} f_{IS}(\x)| \leq C (1 + |\x|)^{|\boldsymbol{u}|} \exp \left( -\left( \frac{1}{2} - A \right) |\x|^2 + B |\x| \right).
\]
\end{lemma}

\begin{proof}
By Leibniz rule, we have
\begin{align*}
\left| \partial^{\boldsymbol{u}} f_{IS} \right| 
&= \left| \sum_{\boldsymbol{u}_1 + \boldsymbol{u}_2 = \boldsymbol{u}} \partial^{\boldsymbol{u}_1} \varphi \partial^{\boldsymbol{u}_2} \frac{f}{q} \right| \\
&= \left| e^{-\frac{|\x|^2}{2}} \sum_{\boldsymbol{u}_1 + \boldsymbol{u}_2 = \boldsymbol{u}} \partial^{\boldsymbol{u}_2} \frac{f}{q} \prod_{j \in \boldsymbol{u}_1} (-x_j) \right| \\
&\leq e^{-\frac{|\x|^2}{2}} \sum_{\boldsymbol{u}_1 + \boldsymbol{u}_2 = \boldsymbol{u}} \left| \partial^{\boldsymbol{u}_2} \frac{f}{q} \right|  |\x|^{|\boldsymbol{u}_1|} \\
&\leq e^{-\frac{|\x|^2}{2}} Ce^{A|\x|^2 + B|\x|}  \sum_{|\boldsymbol{u}_1|} |\x|^{|\boldsymbol{u}_1|} \\
&= C(1 + |\x|)^{|\boldsymbol{u}|} \exp\left\{ -\left(\frac{1}{2} - A\right)|\x|^2 + B|\x| \right\}.
\end{align*}
\end{proof}

\subsection{Proof of Lemma \ref{lemma:V_HK}}
\begin{proof}
By definition of the variation in the sense of Hardy and Krause for smooth functions \cite{niederreiter1992},
\[
V_{HK}(f_{IS} \circ P_R \circ G^{-1}) = \sum_{\varnothing \neq \boldsymbol{u} \subseteq 1:d} \int_{[0,1]^d} \left| \partial^{\boldsymbol{u}} (f_{IS} \circ P_R \circ G^{-1}(\y_{\boldsymbol{u}} : 1_{-\boldsymbol{u}})) \right| .
\]
By Lemma \ref{lemma:DerivativeOfh_IS}, we have
\begin{align} 
&\left| \partial^{\boldsymbol{u}} f \circ P_R \circ G^{-1}(\y_{\boldsymbol{u}} : 1_{-\boldsymbol{u}}) \right| \notag\\
= & \left| \partial^{\boldsymbol{u}} f \left( P_R(\x_{\boldsymbol{u}}) : P_R \circ G^{-1}(1_{-\boldsymbol{u}}) \right) \right| \cdot \left| \prod_{j \in \boldsymbol{u}} \frac{d P_R(G^{-1}(y_j))}{d G^{-1}(y_j)} \frac{d G^{-1}(y_j)}{d y_j} \right| \notag\\
\leq & C (1 + \sqrt{d} R)^{|\boldsymbol{u}|} e^{\sqrt{d} B R}\prod_{j \in \boldsymbol{u}} I_{\{|x_j| \leq R\}} \frac{1}{q(x_j)} ,  \notag
\end{align}
where \( x_j = G^{-1}(y_j)\).
Integrating both side and summing over all \( \boldsymbol{u} \), we obtain that for $ R > 1$,
\begin{align*}
V_{HK}(f_{IS} \circ P_R \circ G^{-1}) 
& \leq \sum_{\boldsymbol{u} \subseteq 1:d} C (1 + \sqrt{d} R)^{|\boldsymbol{u}|} e^{\sqrt{d} B R}(2R)^{|\boldsymbol{u}|} \notag\\
& = C e^{\sqrt{d} B R} \left( 1 + 2 R (1 + \sqrt{d} R) \right)^d \notag\\
&\leq C (2\sqrt{d} + 3)^d R^{2d} e^{\sqrt{d} B R}.
\end{align*}
\end{proof}
\subsection{Proof of Lemma \ref{lemma:ProjectionError}}
\begin{proof}
Let \( H = [-(R-1), R-1]^d \) denote the hypercube in \( \mathbb{R}^d \) with side length \( 2(R-1) \) centered at the origin. Then
\[
\mathbb{E}_q \left| f_{IS}(\boldsymbol{X}) - f_{IS} \circ P_R(\boldsymbol{X}) \right|^2 = \int_{\mathbb{R}^d \setminus H} \left| f_{IS}(\boldsymbol{x}) - f_{IS} \circ P_R(\boldsymbol{x}) \right|^2 q(\boldsymbol{x}) \mathrm{d}\boldsymbol{x}.
\]
Using the mean value theorem, we have
\[
\left| f_{IS}(\boldsymbol{x}) - f_{IS} \circ P_R(\boldsymbol{x}) \right| \leq \left| \nabla f_{IS}(\boldsymbol{\xi}_{\boldsymbol{x}}) \right| \cdot \left| \boldsymbol{x} - P_R (\boldsymbol{x}) \right|,
\]
where \(\boldsymbol{\xi}_{\boldsymbol{x}}\) lies between \(P_R(\boldsymbol{x})\) and \(\boldsymbol{x}\).
By Lemma \ref{lemma:DerivativeOfh_IS}, we have
\[
\left| \nabla f_{IS}(\boldsymbol{\xi}_{\boldsymbol{x}}) \right| \leq \sqrt{d} C (1 + |\boldsymbol{\xi}_{\boldsymbol{x}}|) \exp \left( B |\boldsymbol{\xi}_{\boldsymbol{x}}| \right).
\]
Combing \( |\boldsymbol{\xi}_{\boldsymbol{x}}| = |\lambda P_R(\boldsymbol{x}) + (1-\lambda)\boldsymbol{x}| \leq |\boldsymbol{x}|\), we have
\[
\left| f_{IS}(\boldsymbol{x}) - f_{IS} \circ P_R(\boldsymbol{x}) \right|^2 \leq 4d C^2 (1 + |\boldsymbol{x}|)^2 |\boldsymbol{x}|^2 \exp \left( 2B |\boldsymbol{x}| \right).
\]
Integrating over \( \mathbb{R}^d \setminus H \), we obtain
\begin{align}
\mathbb{E}_q \left| f_{IS}(\boldsymbol{X}) - f_{IS} \circ P_R(\boldsymbol{X}) \right|^2 
& \leq 4d C^2 \int_{\mathbb{R}^d \setminus H} (1 + |\boldsymbol{x}|)^2 |\boldsymbol{x}|^2 \exp \left( 2B |\boldsymbol{x}| \right) q(\boldsymbol{x}) \mathrm{d}\boldsymbol{x} \notag \\
& = 4d C^2\mathbb{E}_q \left[(1+|\boldsymbol{X}|)^2 |\boldsymbol{X}|^2 e^{2B|\boldsymbol{X}|} I_{\mathbb{R}^d \setminus H} \right] \notag\\
& \leq 4d C^2\mathbb{E}_q \left[(1+|\boldsymbol{X}|)^2 |\boldsymbol{X}|^2 e^{2B|\boldsymbol{X}|} I_{\{|\boldsymbol{X}| \geq R - 1\}} \right] \notag\\
& \leq 16d C^2\mathbb{E}_q \left[|\boldsymbol{X}|^4 e^{2B|\boldsymbol{X}|} I_{\{|\boldsymbol{X}| \geq R - 1\}} \right] \notag\\
& \leq 16d C^2 C_0 R^{\alpha} e^{-\beta R^{\gamma}} \notag
\end{align}
for \(R>2\), under the condition, we obtain the desired bound.  

For $1\leq R\leq2$, we note that \( R^{-\alpha}e^{-\beta R^{\gamma}}\mathbb{E}_q \left[(1+|\boldsymbol{X}|)^2 |\boldsymbol{X}|^2 e^{2B|\boldsymbol{X}|} I_{\{|\boldsymbol{X}| \geq R - 1\}} \right]\) as a function of \(R\) is continuous on \(R\in[1,2]\), thus we can find a constant \(C(\alpha,\beta,\gamma)\) such that 
\begin{equation}
\mathbb{E}_q \left[(1+|\boldsymbol{X}|)^2 |\boldsymbol{X}|^2 e^{2B|\boldsymbol{X}|} I_{\{|\boldsymbol{X}| \geq R - 1\}}  \right] \leq C(\alpha,\beta,\gamma) R^{\alpha} e^{-\beta R^{\gamma}}, \notag
\end{equation}
implying that
\begin{equation}
\mathbb{E}_q \left| f_{IS}(\boldsymbol{X}) - f_{IS} \circ P_R(\boldsymbol{X}) \right|^2 \leq 4d C^2 C(\alpha,\beta,\gamma) R^{\alpha} e^{-\beta R^{\gamma}}. \notag
\end{equation}
Therefore, for any \(R > 1\), we can find a constant \(C_1\) such that
\[\mathbb{E}_q \left| f_{IS}(\boldsymbol{X}) - f_{IS} \circ P_R(\boldsymbol{X}) \right|^2 \leq C_1 R^{\alpha} e^{-\beta R^{\gamma}},\]
where \(C_1\) is a constant depending on \(\alpha\), \(\beta\), and \(\gamma\).
This completes the proof.
\end{proof}

\bibliographystyle{plain}
\bibliography{references}

\end{document}